\documentclass[english,11pt,leqno]{amsart}

\usepackage{amsmath,amsfonts,amssymb,amscd,amsthm,amsbsy,upref, mathabx}
\usepackage{color}
 \usepackage{comment}
 \usepackage{amsrefs}
 \usepackage[english]{babel}
 \usepackage{xspace}

\newcommand{\N}{\mathbb{N}}
\newcommand{\R}{\mathbb{R}}

\newcommand{\M}{\mathbb{M}}

\newcommand{\p}{\mathcal{P}}

\newcommand{\A}{\mathcal{A}}

\newcommand{\ie}{\textit{i.e}}

\newcommand{\liminfty}{\lim\limits_{n \to +\infty}}

\newcommand{\cof}{\operatorname{cof}}

\newcommand{\Lip}{\operatorname{Lip}}

\newcommand{\supp}{\operatorname{Supp}}
\newcommand{\dH}{d_{\mathbb{H}}}
\newcommand{\J}{\operatorname{J}}
\newcommand{\nb}{\overline{n}}
\newcommand{\mb}{\overline{m}}
\newcommand{\Np}{\operatorname{N}_p}
\newcommand{\Pp}{\operatorname{P}_p}
\newcommand{\Tp}{\operatorname{T}_p}
\newcommand{\Ap}{\operatorname{A}_p}
\newcommand{\vspan}{\operatorname{span}}

\theoremstyle{plain}
\newtheorem{prop}{Proposition}[section]
\newtheorem{theorem}[prop]{Theorem}
\newtheorem{lemma}[prop]{Lemma}
\newtheorem{corollary}[prop]{Corollary}
\theoremstyle{definition}
\newtheorem{definition}[prop]{Definition}

\newtheorem{pb}{Problem}
\newtheorem*{thank}{Acknowledgments}
\theoremstyle{remark}
\newtheorem{rmk}[prop]{Remark}

\begin{document}

\allowdisplaybreaks

\title[Hamming graphs and concentration properties in Banach spaces]{Hamming graphs and concentration properties in non-quasi-reflexive Banach spaces}

\author{A.~Fovelle}
\address{A.~Fovelle, Laboratoire de Math\'ematiques de Besan\c con, Universit\'e Bourgogne Franche-Comt\'e, 16 route de Gray, 25030 Besan\c con C\'edex, Besan\c con, France}
\email{audrey.fovelle@univ-fcomte.fr}

\thanks{The LmB receives support from the EIPHI Graduate School (contract ANR-17-EURE-0002)}


\begin{abstract} In this note, we study some concentration properties for Lipschitz maps defined on Hamming graphs, as well as their stability under sums of Banach spaces. As an application, we extend a result of Causey on the coarse Lipschitz structure of quasi-reflexive spaces satisfying upper $\ell_p$ tree estimates to the setting of $\ell_p$-sums of such spaces. Our result provides us with a tool for constructing the first examples of Banach spaces that are not quasi-reflexive but nevertheless admit some concentration inequality. We also give a sufficient condition for a space to be asymptotic-$c_0$ in terms of a concentration property, as well as relevant counterexamples.
\end{abstract}

\maketitle

 \setcounter{tocdepth}{1}

\section{Introduction}

In 2008, in order to show that $L_p(0,1)$ is not uniformly homeomorphic to $\ell_p \oplus \ell_2$ for $p \in (1, \infty) \setminus \{2 \}$, Kalton and Randrianarivony \cite{KR} introduced a new technique based on a certain class of graphs and asymptotic smoothness ideas. To be more specific, they introduced a concentration property for Lipschitz maps defined on Hamming graphs into a reflexive asymptotically uniformly smooth (AUS) Banach space $X$ (we refer the reader to Section $2$ for the definitions), that prevents coarse embeddings of certain other spaces into $X$. Their result was used by Kalton himself to deduce some information about the spreading models of a space that coarse Lipschitz embeds into a reflexive AUS space (see \cite{MR2995383}), and  was later extended to the quasi-reflexive case by Lancien et Raja \cite{LR}, who introduced a weaker concentration property. Soon after, Causey \cite{3.5} proved that this same weaker concentration property also applies to quasi-reflexive spaces with so-called upper $\ell_p$ tree estimates.

The purpose of this paper is to start a general study of these concentration properties, together with new ones. In particular, we will address the question of their stability under sums of Banach spaces. This will allow us to get non-quasi-reflexive examples. \\
Our main result will be the following. 

\begin{theorem} \label{thm somme E}
Let $p \in (1, \infty)$, $\lambda > 0$, $(X_n)_{n \in \N}$ a sequence of Banach spaces with property $\lambda$-HIC$_{p,d}$. \\
Let $E$ be a reflexive Banach space with a normalized $1$-unconditional $p$-convex basis $(e_n)_{n \in \N}$ with convexity constant $1$. \\
Then $X= \left( \sum_{n \in \N} X_n \right)_E$ has property $(\lambda+2+\varepsilon)$-HIC$_{p,d}$ for every $\varepsilon > 0$.
\end{theorem}

where property $\lambda$-HIC$_{p,d}$ is a refinement of the property $\lambda$-HIC$_p$ first considered by Lancien and Raja: a space $X$ has property $\lambda$-HIC$_p$ if for any Lipschitz function $f : ([\N]^k, \dH) \to X$, there exist $\nb, \mb$ in interlacing position such that $\|f(\nb)-f(\mb)\| \leq \lambda k^{1/p} \Lip(f)$ (see Section 2 for the definitions of the concentration properties). 

As a consequence, we get that an $\ell_q$-sum of a quasi-reflexive Banach space satisfying upper $\ell_p$ tree estimates, $1 < q,p < \infty$, cannot equi-Lipschitz contain the Hamming graphs. This is a generalization of the result mentionned above by Causey (see \cite{3.5}), who proved it for quasi-reflexive Banach spaces satisfying upper $\ell_p$ tree estimates. This is the first result of this type for non-quasi-reflexive Banach spaces.

In order to show this result, we introduce the notions and the terminology we will use later in the second section of this paper while Section $3$ is dedicated to the proof itself.

Recently, Baudier, Lancien, Motakis and Schlumprecht \cite{Baudier2020ANC} proved that any quasi-reflexive asymptotic-$c_0$ Banach space $X$ (see Section $4$ for the definition of asymptotic-$c_0$) has property HIC$_{\infty}$. Even though we don't know if a Banach space with this property is quasi-reflexive, we prove in the fourth and last section of this paper that property HIC$_{\infty}$ implies asymptotic-$c_0$. In particular, the space $T^*(T^*)$, where $T^*$ is the original Banach space constructed by Tsirelson in \cite{Tsirelson}, cannot have this concentration property. \\
We also give a striking example in the non-quasi-reflexive setting of a separable dual asymptotic-$c_0$ space that does not Lipschitz contain $\ell_1$ nor $c_0$ and without any of the concentration properties introduced in this paper. This example is based on a generalization of the construction of Lindenstrauss spaces, that we owe to Schlumprecht. This construction is detailed in the last section.

\section{Definitions and notation}

\subsection{Basic definitions}
All Banach spaces in these notes are assumed to be real and infinite-dimensional unless otherwise stated. We denote the closed unit ball of a Banach space $X$ by $B_X$, and its unit sphere by $S_X$. Given a Banach space $X$ with norm $\|\cdot\|_X$, we simply write $\|\cdot\|$ as long as it is clear from the context on which space it is defined. We recall that a Banach space is said to be \textit{quasi-reflexive} if the image of its canonical embedding into its bidual is of finite codimension in its bidual. \\
We say that a basic sequence $(e_i)_{i \in \N}$ of a Banach space $E$ is \textit{$c$-unconditional}, for some $c \geq 1$, if, for any $(a_i)_{i \in \N}, (b_i)_{i \in \N} \in c_{00}$ (the vector space of all real sequences with finite support), we have :
\[ \Big\| \sum_{i=1}^{\infty} a_i e_i \Big\| \leq c \Big\| \sum_{i=1}^{\infty} b_i e_i \Big\|  \]
whenever $|a_i|\leq |b_i|$ for every $i \in \N$. \\
Let $(X_n)_{n \in \N}$ be a sequence of Banach spaces. Let $\mathcal{E}=(e_n)_{n \in \N}$ be a $1$-unconditional basic sequence in a Banach space $E$ with norm $\|\cdot\|_E$. We define the sum $\left( \sum_{n \in \N} X_n \right)_{\mathcal{E}}$ to be the space of sequences $(x_n)_{n \in \N}$, where $x_n \in X_n$ for all $n \in \N$, such that $\sum_{n \in \N} \|x_n\|_{X_n} e_n$ converges in $E$, and we set
\[ \|(x_n)_{n \in \N} \| = \Big\| \sum_{n \in \N} \|x_n\|_{X_n} e_n \Big\|_E < \infty . \]
One can check that $\left( \sum_{n \in \N} X_n \right)_{\mathcal{E}}$, endowed with the norm $\|\cdot\|$ defined above, is a Banach space. We can, in a similar way, define finite sums $\left( \sum_{j=1}^n X_j \right)_{\mathcal{E}}$ for all $n \in \N$, and, in case $n=2$, we will write $X_1 \bigoplus\limits_{\mathcal{E}} X_2$. If it is implicit what is the basis $\mathcal{E}$ of the Banach space $E$ that we are working with, we write $\left( \sum_{n \in \N} X_n \right)_E$ or $X_1 \bigoplus\limits_E X_2$. Also, if the $X_n$'s are all the same, say $X_n=X$, for all $n \in \N$, we write $E(X)$. 

Let us finish this section with the following definition. \\
Let $p \in (1,  \infty )$ and $E$ be a Banach space with a $1$-unconditional basis $(e_n)_{n \in \N}$. We say that the basis $(e_n)_{n \in \N}$ is \textit{$p$-convex with convexity constant $C$} if :
\[ \Big\| \sum\limits_{j \in \N} (|x^1_j|^p + \cdots + |x_j^k|^p)^{\frac{1}{p}} e_j \Big\|^p \leq C^p \sum_{n=1}^k \|x^n\|^p \]
for all $x^1=\sum\limits_{j=1}^{\infty} x_j^1 e_j$, $\cdots$, $x^k=\sum\limits_{j=1}^{\infty} x_j^k e_j \in E$ (cf Definition 1.d.3 of \cite{LT2}).

\subsection{Hamming graphs}

Before introducing concentration properties, we need to define special metric graphs that we shall call \textit{Hamming graphs}. Let $\M$ be an infinite subset of $\N$. We denote by $[\M]^{\omega}$ the set of infinite subsets of $\M$. For $\M \in [\N]^{\omega}$ and $k \in \N$, let
\[ [\M]^k = \{ \overline{n}=(n_1, \dots, n_k) \in \M^k ; n_1 < \cdots < n_k \} , \]
\[ [\M]^{\leq k} = \bigcup_{j=1}^k [\M]^j \cup \{ \varnothing \} , \]
and
\[ [\M]^{< \omega} = \bigcup\limits_{k=1}^{\infty} [\M]^k \cup \{ \varnothing \} . \]
Then we equip $[\M]^k$ with the \textit{Hamming distance}:
\[ \dH(\overline{n}, \overline{m})=| \{ j ; n_j \neq m_j \} |  \]
for all $\overline{n}=(n_1, \dots, n_k), \overline{m}=(m_1, \dots, m_k) \in [\M]^k$. \\
Let us mention that this distance can be extended to $[\M]^{< \omega}$ by letting
\[ \dH(\overline{n}, \overline{m})=| \{ i \in \{1, \cdots, \min(l,j) \} ; n_i \neq m_i \} | + \max(l,j)-\min(l,j)  \]
for all $\overline{n}=(n_1, \dots, n_l), \overline{m}=(m_1, \dots, m_j) \in [\M]^{< \omega}$ (with possibly $l=0$ or $j=0$). \\
We also need to introduce $I_k(\M)$, the set of strictly interlaced pairs in $[\M]^k$:
\[ I_k(\M)= \{ (\overline{n}, \overline{m}) \subset [\M]^k ; n_1 < m_1 < \cdots < n_k < m_k  \} \]
and, for each $j \in \{ 1, \cdots, k \}$, let
\[ H_j(\M)= \{ (\overline{n}, \overline{m}) \subset [\M]^k ; \forall i \neq j, n_i=m_i \text{ and } n_j < m_j \} . \]
Note that, for $(\overline{n}, \overline{m}) \in I_k(\M)$, $\dH(\overline{n}, \overline{m})=k$ and $\overline{n} \cap \overline{m} = \varnothing$. \\

Let us mention that, in this paper, we will only be interested in the Hamming distance but originally, when Hamming graphs were used in \cite{KR}, it could be equally replaced (unless for their last Theorem 6.1) by the \textit{symmetric distance}, defined by
\[ d_{\Delta}(\overline{n}, \overline{m})=\dfrac{1}{2} | \overline{n} \triangle \overline{m} | \] 
for all $\overline{n}, \overline{m} \in [\N]^{< \omega}$, where $\overline{n} \triangle \overline{m}$ denotes the symmetric difference between $\overline{n}$ and $\overline{m}$.

\subsection{Asymptotic properties}
We now define the uniform asymptotic properties of norms that will be considered in this paper. Let $(X, \|\cdot\|)$ be a Banach space. Following Milman (see \cite{milman}), we introduce the two following moduli: for all $t \geq 0$, let
\[ \overline{\rho}_X(t)= \sup\limits_{x \in S_X} \inf\limits_{Y} \sup\limits_{y \in S_Y} (\|x+ty\|-1) \text{ and } \overline{\delta}_X(t)= \inf\limits_{x \in S_X} \sup\limits_{Y} \inf\limits_{y \in S_Y} (\|x+ty\|-1) \]
where $Y$ runs through all closed linear subspaces of $X$ of finite codimension. \\
We say that $\|\cdot\|$ is \textit{asymptotically uniformly smooth} (in short AUS) if $\lim_{t \to 0} \frac{\overline{\rho}_X(t)}{t}=0$. We say that $\|\cdot\|$ is \textit{asymptotically uniformly convex} (in short AUC) if $\overline{\delta}_X(t)>0$ for all $t >0$. If $p \in (1, \infty )$, $\|\cdot\|$ is said to be \textit{$p$-AUS} if there is a constant $C > 0$ such that, for all $t \in [0, \infty )$, $\overline{\rho}_X(t) \leq C t^p$. If $q \in [1, \infty )$, $\|\cdot\|$ is said to be \textit{$q$-AUC} if there is a constant $C > 0$ such that, for all $t \in [0, 1 ]$, $\overline{\delta}_X(t) \geq C t^q$. If $X$ has an equivalent norm for which $X$ is AUS (resp. $p$-AUS), $X$ is said to be \textit{AUSable} (resp. \textit{$p$-AUSable}). Every asymptotically uniformly smooth Banach space is $p$-AUSable for some $p \in (1, \infty)$, this was first proved for separable Banach spaces by Knaust, Odell and Schlumprecht (see \cite{knaust1999asymptotic}) and later generalized by Raja for any Banach space (see \cite{raja}, Theorem 1.2). 

Let $X$ be a Banach space, $B \subset X$ and $\M \in [\N]^{\omega}$. A family $(x_{\nb})_{\nb \in [\M]^{\leq k}}$ in $B$ is \textit{a tree in $B$} of \textit{height} $k$. This tree $(x_{\nb})_{\nb \in [\M]^{\leq k}}$ is said to be \textit{weakly null} if the sequence $(x_{\nb,n})_{n>\max(\nb)}$ is weakly null for every $\nb \in [\M]^{k-1} \cup \{\varnothing \}$ (with $\max(\varnothing)=0$).

Let $1 < p \leq \infty$, $C >0$. We say that $X$ satisfies \textit{upper $\ell_p$ tree estimates with constant $C$} if for any $k \in \N$ and any weakly null tree $(x_{\overline{n}})_{\overline{n} \in [\N]^{ \leq k}}$ in  $B_X$, there exists $\overline{n} \in [\N]^k$ such that
\[ \forall a=(a_1, \dots, a_k) \in \R^k, \Big\| \sum_{j=1}^k a_j x_{(n_1, \dots, n_j)} \Big\| \leq C \|a\|_{\ell_p^k} . \]
We say that $X$ satisfies \textit{upper $\ell_p$ tree estimates} if $X$ satisfies upper $\ell_p$ tree estimates with constant $C$ for some $C>0$. 

We say that $X$ has the \textit{tree-$p$-Banach-Saks property with constant $C$} if for any $k \in \N$ and any weakly null tree $(x_{\overline{n}})_{\overline{n} \in [\N]^{ \leq k}}$ in  $B_X$, there exists $\overline{n} \in [\N]^k$ such that
\[ \Big\| \sum_{j=1}^k x_{(n_1, \dots, n_j)} \Big\| \leq Ck^{1/p}  \]
with the convention $1/p=0$ if $p= \infty$. \\
We say that $X$ has the \textit{tree-$p$-Banach-Saks property} (tree-$p$-BS) if $X$ has the tree-$p$-Banach-Saks property with constant $C$ for some $C >0$. 

Let $p \in (1, \infty)$. We denote by $\Tp$ the class of all $p$-AUSable Banach spaces, by $\Ap$ the class of all Banach spaces satisfying upper $\ell_p$ tree estimates, by $\Np$ the class of all Banach spaces with tree-$p$-BS property and $\Pp=\bigcap_{1<r<p} \operatorname{T}_r$. It is known that (see \cite{GKL} and \cite{Braga2})
\[ \Tp \subset \Ap \subset \Np \subset \Pp  \]
and Causey proved in \cite{3.5} (cf Theorem 6.2) that these inclusions are strict (even among reflexive spaces). \\
It is also known (see \cite{Braga2}) that $\operatorname{N}_{\infty}=\operatorname{A}_{\infty}$.

\subsection{Metric embeddings}

Let us recall some definitions on metric embeddings. \\
Let $(X,d_X)$ and $(Y, d_Y)$ two metric spaces, $f$ a map from $X$ to $Y$. \\
We define the \textit{compression modulus} of $f$ by 
\[ \forall t \geq 0, \hspace{2mm} \rho_f(t)= \inf \{ d_Y(f(x),f(y)) ; d_X(x,y) \geq t \} ; \]
and the \textit{expansion modulus} of $f$ by 
\[ \forall t \geq 0, \hspace{0.2cm} \omega_f(t)= \sup \{ d_Y(f(x),f(y)) ; d_X(x,y) \leq t \} . \]
We adopt the convention $\inf(\varnothing)= + \infty$. Note that, for every $x,y \in X$, we have 
\[ \rho_f(d_X(x,y)) \leq d_Y(f(x),f(y)) \leq \omega_f(d_X(x,y)) . \]
We say that $f$ is a \textit{bi-Lipschitz embedding} if there exist $A,B$ in $(0, \infty )$ such that $\rho_f(t) \geq At$ and $\omega_f(t) \leq Bt$ for all $t \geq 0$. If there exists such an embedding $f$, we denote $(X, d_X) \underset{L}{\hookrightarrow} (Y, d_Y)$. \\
If the metric spaces are unbounded, the map $f$ is said to be a \textit{coarse embedding} if $\lim_{t \to \infty} \rho_f(t)= \infty$ and $\omega_f(t) < \infty$ for all $t > 0$. \\
If one is given a family of metric spaces $(X_i)_{i \in I}$, one says that $(X_i)_{i \in I}$ \textit{equi-Lipschitz embeds} into $Y$ if there exist $A, B$ in $(0, \infty )$ and, for all $i \in I$, maps $f_i : X_i \to Y$ such that $\rho_{f_i}(t) \geq At$ and $\omega_{f_i}(t) \leq Bt$ for all $t \geq 0$. One also says that the family $(X_i)_{i \in I}$ equi-coarsely embeds into $Y$ if there exist non-decreasing functions $\rho, \omega : [0, \infty ) \to [0, \infty )$ and for all $i \in I$, maps $f_i : X_i \to Y$ such that $\rho \leq \rho_{f_i}$, $\omega_{f_i} \leq \omega$, $\lim_{t \to \infty} \rho(t)= \infty$ and $\omega(t) < \infty$ for all $t > 0$. \\
Besides, we say that $f$ is a \textit{coarse Lipschitz embedding} if there exist $A, B, C, D$ in $(0, \infty)$ such that $\rho_f(t) \geq At-C$ and $\omega_f(t) \leq Bt+D$ for all $t \geq 0$. If $X$ and $Y$ are Banach spaces, this is equivalent to the existence of numbers $\theta \geq 0$ and $0 < c_1 < c_2$ so that :
\[ c_1 \|x-y\|_X \leq \|f(x)-f(y)\|_Y \leq c_2 \|x-y\|_X \]
for all $x,y \in X$ satisfying $\|x-y\|_X \geq \theta$.  \\
Finally, a way to refine the scale of coarse embeddings is to talk about compression exponents, introduced by Guentner and Kaminker in \cite{MR2160829}. Let $X$ and $Y$ to Banach spaces. The \textit{compression exponent of $X$ in $Y$}, denoted by $\alpha_Y(X)$, is the supremum of all $\alpha \in [0,1)$ for which there exist a coarse embedding $f :X \to Y$ and $A, C$ in $(0, \infty )$ so that $\rho_f(t) \geq A t^{\alpha}-C$ for all $t > 0$.

\subsection{Definitions of concentration properties}

In this subsection, we introduce all the concentration properties mentioned in this paper. Before doing so, let us recall a version of Ramsey's Theorem we will use several times.

\begin{theorem}[Ramsey's Theorem \cite{Ramsey}] 
Let $k \in \N$ and $\A \subset [\N]^k$. \\
There exists $\M \in [\N]^{\omega}$ such that either $[\M]^k \subset \A$ or $[\M]^k \cap \A = \varnothing$.
\end{theorem}

The following properties are studied in \cite{KR}, \cite{LR} and \cite{Baudier2020ANC}. We use the convention $1/ \infty =0$.

\begin{definition}
Let $(X,d)$ be a metric space, $\lambda > 0$, $p \in (1, \infty ]$. \\
$\bullet$ We say that $X$ has property \textit{$\lambda$-HFC$_p$} (Hamming Full Concentration) if, for all $k \in \N$, for every Lipschitz function $f : ([\N]^k, \dH) \to X$, one can find $\M \in [\N]^{\omega}$ such that
\[ \forall \overline{n}, \overline{m} \in [\M]^k, \hspace{2mm}  d(f(\overline{n}),f(\overline{m})) \leq \lambda k^{\frac{1}{p}} \Lip(f) . \]
We say that $X$ has property \textit{HFC$_p$} if $X$ has property $\lambda$-HFC$_p$ for some $\lambda > 0$. \\
$\bullet$ We say that $X$ has property \textit{$\lambda$-HIC$_p$} (Hamming Interlaced Concentration) if, for all $k \in \N$, for every Lipschitz function $f : ([\N]^k, \dH) \to X$, one can find $(\overline{n}, \overline{m}) \in I_k(\N)$ satisfying
\[  d(f(\overline{n}),f(\overline{m})) \leq \lambda k^{\frac{1}{p}} \Lip(f) . \]
We say that $X$ has property \textit{HIC$_p$} if $X$ has property $\lambda$-HIC$_p$, for some $\lambda > 0$. \\
\end{definition}

\begin{rmk}
$1)$ Let us notice that, by Ramsey's Theorem ($I_k(\N)$ can be identified with $[\N]^{2k}$), a metric space $(X,d)$ has property $\lambda$-HIC$_p$ if and only if, for all $k \in \N$, for every Lipschitz function $f : ([\N]^k, \dH) \to X$, one can find $\M \in [\N]^{\omega}$ that satisfies
\[ \forall (\overline{n}, \overline{m}) \in I_k(\M), \hspace{2mm} d(f(\overline{n}),f(\overline{m})) \leq \lambda k^{\frac{1}{p}} \Lip(f) . \]
$2)$ Baudier, Lancien, Motakis and Schlumprecht showed that property HFC$_{\infty}$ is equivalent for a Banach space to being reflexive and asymptotic-$c_0$ (see \cite{Baudier2020ANC} for the proof of this result and Section \ref{section asympt} for the definition of asymptotic-$c_0$).
\end{rmk}

We now introduce a property that seems weaker than the previous one but is enough to prevent the equi-Lipschitz embedding (or equi-coarse embedding for the case $p= \infty$) of Hamming graphs. We will show later that this property actually coincides with property HIC$_p$, $p \in (1, \infty ]$.

\begin{definition}
Let $(X,d)$ be a metric space, $\lambda > 0$, and $p \in (1, \infty ]$. We say that $X$ has property \textit{$\lambda$-HC$_p$} if, for all $k \in \N$, for every Lipschitz function $f : ([\N]^k, \dH) \to X$, one can find $\overline{n}, \overline{m} \in [\N]^{k}$ satisfying $\overline{n} \cap \overline{m} = \varnothing$ and
\[  d(f(\overline{n}),f(\overline{m})) \leq \lambda k^{\frac{1}{p}} \Lip(f) . \]
We say that $X$ has property \textit{HC$_p$} if $X$ has property $\lambda$-HC$_p$, for some $\lambda > 0$.
\end{definition}

It is easy to check that all these concentration properties are stable under coarse Lipschitz embeddings and that properties HFC$_{\infty}$, HC$_{\infty}$ and HIC$_{\infty}$ are even stable under coarse embeddings, when the embedded space is a Banach space. 

Let us now introduce the last concentration properties we will study here, more precise than HC$_p$ and HIC$_p$, $p \in (1, \infty )$, where directional Lipschitz constants take part, hence the ``$d$" in subscript in the acronyms below.

\begin{definition}
Let $(X,d)$ be a metric space, $\lambda > 0$, $p \in (1, \infty)$. \\
$\bullet$ We say that $X$ has property \textit{$\lambda$-HFC$_{p,d}$} (resp. \textit{$\lambda$-HIC$_{p,d}$}) if, for every $k \in \N$ and every Lipschitz function $f : ([\N]^k, \dH) \to X$, there exists $\M \in [\N]^{\omega}$ such that
\[ d(f(\overline{n}),f(\overline{m}) ) \leq \lambda \left( \sum_{j=1}^k \alpha_j^p \right)^{\frac{1}{p}} \]
for all $\overline{n}, \overline{m} \in [\M]^k$ (resp. $(\overline{n}, \overline{m}) \in I_k(\M)$), where, for each $j \in \{1, \cdots, k\}$
\[ \alpha_j = \sup\limits_{(\overline{n}, \overline{m}) \in H_j(\N)} d(f(\overline{n}),f(\overline{m})) . \]
We say that $X$ has property \textit{HFC$_{p,d}$} (resp. \textit{HIC$_{p,d}$}) if $X$ has property $\lambda$-HFC$_{p,d}$ (resp. $\lambda$-HIC$_{p,d}$), for some $\lambda > 0$. \\
$\bullet$ Similary, we say that $X$ has property \textit{$\lambda$-HC$_{p,d}$} if, for every $k \in \N$ and every Lipschitz function $f : ([\N]^k, \dH) \to X$, one can find $\overline{n}, \overline{m} \in [\N]^k$ satisfying $\overline{n} \cap \overline{m} = \varnothing$ and
\[  d(f(\overline{n}),f(\overline{m})) \leq \lambda \left( \sum_{j=1}^k \alpha_j^p \right)^{\frac{1}{p}} \]
where the $\alpha_j$, $j \in \{1, \cdots, k\}$, are defined as above. \\
We say that $X$ has property \textit{HC$_{p,d}$} if $X$ has property $\lambda$-HC$_{p,d}$, for some $\lambda > 0$.
\end{definition}

It is important to note that Theorem 6.1 \cite{KR} and Theorem 5.2 \cite{3.5} can be rephrased as follows: for $p \in (1, \infty)$, a reflexive (resp. quasi-reflexive) Banach space satisfying upper $\ell_p$ tree estimates has property HFC$_{p,d}$ (resp. HC$_{p,d}$) and a reflexive (resp. quasi-reflexive) Banach space with tree-$p$-Banach-Saks property has property HFC$_p$ (resp. HC$_p$). Even though Kalton and Randrianarivony \cite{KR} proved their theorem for reflexive $p$-AUS Banach spaces, their proof implicitely contains the latter result. Let us also note that a Banach space with property HFC$_p$ is necessarily reflexive (see \cite{BKL}). In 2017, Lancien and Raja \cite{LR} proved that all quasi-reflexive $p$-AUS Banach spaces have property HC$_{p,d}$. It was later extended as mentionned by Causey \cite{3.5}.

The stability of these last properties under coarse Lipschitz embeddings when the embedded space is a Banach space is a bit less clear so we include a proof for completeness. 

\begin{prop}
Let $p \in (1, \infty )$, $P \in \{$HFC$_{p,d}$, HIC$_{p,d}$, HC$_{p,d}\}$, $X$ a Banach space and $(Y,d_Y)$ a metric space. \\
If $Y$ has property $P$ and $X$ coarse Lipschitz embeds into $Y$, then $X$ has property $P$.
\end{prop}

\begin{proof}
We only prove the stability of HFC$_{p,d}$, the proofs for the other two properties are similar. \\
Let us assume that $Y$ has property $\lambda$-HFC$_{p,d}$ for a $\lambda > 0$ and that there exist a map $\varphi : X \to Y$ and $A, B, C, D > 0$ such that $\rho_{\varphi}(t) \geq At-B$ and $\omega_{\varphi}(t) \leq Ct+D$ for all $t \geq 0$. \\
Let $k \in \N$, $f : ([\N]^k, \dH) \to X$ a Lipschitz function with $\Lip(f) > 0$. \\
Without loss of generality, we can assume that, for all $j \in \{1, \cdots,k\}$, we have
\[ \alpha_j = \sup\limits_{(\overline{n}, \overline{m}) \in H_j(\M)} \| f(\overline{n})-f(\overline{m}) \| > 0 . \] 
Indeed, $\dH$ is a graph metric so $\max_{j \in \{1, \cdots, k\}} \alpha_j = \Lip(f) > 0$ and if $\alpha_j=0$ for some $j \in \{1, \cdots, k\}$, then the expression of $f$ does not depend on this $j^{\text{\textsuperscript{th}\xspace}}$ coordinate. \\
Therefore \[ \alpha = \min\limits_{1 \leq j \leq k} \alpha_j \in ( 0, \Lip(f) ] . \] 
Let us note that $\omega_{\varphi}(t) \leq (C+D) t$ for all $t \geq 1$ so, for all $j \in \{1, \cdots, k \}$ and for all $(\overline{n}, \overline{m}) \in H_j(\N)$, we have
\[ d_Y \left( \varphi \left( \frac{1}{\alpha} f(\overline{n}) \right) , \varphi \left( \frac{1}{\alpha} f(\overline{m}) \right) \right) \leq \omega_{\varphi} \left( \frac{\alpha_j}{\alpha} \right) \leq \frac{C+D}{\alpha} \alpha_j . \]
Now, by assumption on $Y$, we can find $\M \in [\N]^{\omega}$ so that
\[ d_Y \left( \varphi \left( \frac{1}{\alpha} f(\overline{n}) \right) , \varphi \left( \frac{1}{\alpha} f(\overline{m}) \right) \right) \leq \frac{\lambda(C+D)}{\alpha} \left( \sum_{j=1}^k \alpha_j^p \right)^{\frac{1}{p}} \]
for all $\overline{n}, \overline{m} \in [\M]^k$. \\
Thus 
\[ \|f(\overline{n})-f(\overline{m}) \| \leq \frac{\lambda(C+D)}{A} \left( \sum_{j=1}^k \alpha_j^p \right)^{\frac{1}{p}} + \frac{\alpha B}{A} \leq \frac{\lambda(C+D)+B}{A} \left( \sum_{j=1}^k \alpha_j^p \right)^{\frac{1}{p}} \]
for all $\overline{n}, \overline{m} \in [\M]^k$. Consequently, $X$ has property HFC$_{p,d}$.
\end{proof}
 
As promised, the next proposition shows that properties HC$_p$ and HIC$_p$, $p \in (1, \infty ]$, are equivalent. This explains why we will only talk about property HC$_{\infty}$ in the last section. \\
Before proving this result, let us introduce some vocabulary. Let $\M \in [\N]^{\omega}$. For $\overline{n}, \overline{m} \in [\M]^k$ satisfying $\overline{n} \cap \overline{m} = \varnothing$, we denote by $\phi$ the unique increasing bijection from $\overline{n} \cup \overline{m}$ onto $\{1, \cdots, 2k \}$. If
\[ I = \left\{ A \subset \{1, \cdots, 2k \} ; |A|=k \right\}, \]
we say that $(\overline{n}, \overline{m})$ is in position $A \in I$ if $\phi(\overline{n})=A$. \\
Thus, we note that the pair $(\nb, \mb)$ with $\overline{n}, \overline{m} \in [\M]^k$ and $\overline{n} \cap \overline{m} = \varnothing$, can be in $\binom{2k-1}{k-1}$ possible different positions if we ask $n_1$ to be the first element (and we can do it without loss of generality). We denote these positions by $\p_i^k(\M)$, $i \in \left\{1, \cdots, \binom{2k-1}{k-1} \right\}$. Let us remark that each one of these positions can be identified with $[\M]^{2k}$, which will allow us to use Ramsey's Theorem.

\begin{prop} \label{prop KRI=K}
For every $p \in (1, \infty ]$, properties HC$_p$ and HIC$_p$ are equivalent. More precisely, a metric space with property $\lambda$-HIC$_p$, for some $\lambda > 0$, has property $\lambda$-HC$_p$ and a metric space with property $\lambda$-HC$_p$ has property $2 \lambda$-HIC$_p$.
\end{prop}

\begin{proof}
For every $p \in (1, \infty ]$, $\lambda >0$, the implication $\lambda$-HIC$_p \implies$ $\lambda$-HC$_p$ is clear so let us show the other implication. \\
We will do it with $p = \infty$, the other cases can be treated similarly. \\
Let $(X,d)$ be a metric space with property $\lambda$-HC$_{\infty}$ for some $\lambda >0$. Let $k \in \N$, $f : ([\N]^k, d_{\mathbb{H}}) \to X$ a Lipschitz function. \\
For each $\M \in [\N]^{\omega}$, there exist $i \in \left\{1, \cdots, \binom{2k-1}{k-1} \right\}$ and $(\overline{n}, \overline{m}) \in \p_i^k(\M)$ such that $d(f(\overline{n}),f(\overline{m})) \leq \lambda \Lip(f)$. \\
Let us show that there exist $i \in \left\{1, \cdots, \binom{2k-1}{k-1} \right\}$ and $\M \in [\N]^{\omega}$ such that $d(f(\overline{n}),f(\overline{m})) \leq \lambda \Lip(f)$ for all $(\overline{n}, \overline{m}) \in \p_i^k(\M)$. \\
By Ramsey's Theorem, if $\mathcal{A}_1= \{ (\overline{n}, \overline{m}) \in \p_1^k(\N) ; d(f(\overline{n}),f(\overline{m})) \leq \lambda \Lip(f) \} \subset \p_1^k(\N)$, there exists $\M_1 \in [\N]^{\omega}$ such that $\p_1^k(\M_1) \subset \mathcal{A}_1$ or $\p_1^k(\M_1) \cap \mathcal{A}_1 = \varnothing$. \\
If $\p_1^k(\M_1) \cap \mathcal{A}_1 = \varnothing$, we apply the same result with $\mathcal{A}_2= \{ (\overline{n}, \overline{m}) \in \p_2^k(\M_1) ; d(f(\overline{n}),f(\overline{m})) \leq \lambda \Lip(f) \} \subset \p_2^k(\M_1)$ and we get $\M_2 \in [\M_1]^{\omega}$ such that $\p_2^k(\M_2) \subset \mathcal{A}_2$ or $\p_2^k(\M_2) \cap \mathcal{A}_2 = \varnothing$. \\
We continue this way inductively. \\ 
As $X$ has property $\lambda$-HC$_{\infty}$, we cannot repeat this operation for all $\binom{2k-1}{k-1}$ positions so there exist $i \in \left\{1, \cdots, \binom{2k-1}{k-1} \right\}$ and $\M \in [\N]^{\omega}$ such that $d(f(\overline{n}),f(\overline{m})) \leq \lambda \Lip(f)$ for all $(\overline{n}, \overline{m}) \in \p_i^k(\M)$. \\
Let us show that there exists $(\overline{n}, \overline{m}) \in I_k(\N)$ such that $d(f(\overline{n}),f(\overline{m})) \leq 2 \lambda \Lip(f)$. \\
For that, let $\M=\{q_1 < q_2 < \cdots < q_j < \cdots \}$. \\
Now, we just have to observe that we can choose $(\overline{n}, \overline{p}) \in \p_i^k(\M)$ such that $n_1 < p_1$ and $\overline{n}, \overline{p} \subset \{ q_1, q_{2k+1}, \cdots, q_{2k(2k-1)+1} \}$. This leaves us enough space to get an element $\overline{m} \in [\M]^k$ so that $(\overline{n}, \overline{m}) \in I_k(\M)$ and $(\overline{m}, \overline{p}) \in \p_i^k(\M)$. \\
The result follows from the triangle inequality.
\end{proof}

\begin{rmk} \label{rmq KRI=K}
With a similar proof, we can prove that properties HC$_{p,d}$ and HIC$_{p,d}$ are equivalent.
\end{rmk}

\section{Stability under sums} \label{section sums}
\subsection{Statements}

In order to prove the stability of property HC$_{p,d}$, $p \in (1, \infty)$, under $\ell_p$ sums, the idea is to adapt Braga's proof of Proposition 7.2 in \cite{Braga2} with property HIC$_{p,d}$ instead of property $p$-Banach-Saks. \\
To do so, we need the following proposition. We chose to state it with property HIC$_{p,d}$, which we recall is equivalent to property HC$_{p,d}$, but the same result can be shown for property HFC$_{p,d}$ with a similar proof.

\begin{prop} \label{somme 2 esp}
Let $p \in (1, + \infty)$, $\lambda > 0$, $E$ be a Banach space with a normalized $1$-unconditional $p$-convex basis $(e_n)_{n \in \N}$ with convexity constant $1$. \\
For every $n \in \N$ and every finite sequence $(X_j)_{j=1}^n$ of Banach spaces having property $\lambda$-HIC$_{p,d}$, the space $\left( \sum_{j=1}^n X_j \right)_E$ has property $(\lambda +\varepsilon)$-HIC$_{p,d}$ for each $\varepsilon > 0$.
\end{prop}

\begin{proof}
It is enough to prove this result for $X=X_1 \bigoplus\limits_{E} X_2$. \\
Let $k \in \N$, $\M \in [\N]^{\omega}$, $\varepsilon > 0$, $h=(f,g) : ([\M]^k,\dH) \to X$ a Lipschitz function. \\
For each $j \in \{1, \cdots, k\}$, let $\gamma_j = \sup\limits_{(\overline{n}, \overline{m}) \in H_j(\M)} \|h(\overline{n})-h(\overline{m})\|$. \\
There exists $\varepsilon' > 0$ such that 
\[ \lambda^p \sum\limits_{j=1}^k (\gamma_j+2 \varepsilon')^p \leq (\lambda+\varepsilon)^p \sum\limits_{j=1}^k \gamma_j^p . \]
Let $\alpha_1=\inf\limits_{ \M_1 \in [\M]^{\omega}} \sup\limits_{(\overline{n},\overline{m}) \in H_1(\M_1)} \|f(\overline{n})-f(\overline{m})\|$. \\
There exists $\M_1 \in [\M]^{\omega}$ so that $\|f(\overline{n})-f(\overline{m})\| \leq \alpha_1+ \varepsilon'$ for every $(\overline{n},\overline{m}) \in H_1(\M_1)$. \\
Let $\beta_1=\inf\limits_{ \M_1' \in [\M_1]^{\omega}} \sup\limits_{(\overline{n},\overline{m}) \in H_1(\M_1')} \|g(\overline{n})-g(\overline{m})\|$. \\
There exists $\M_1' \in [\M_1]^{\omega}$ so that $\|g(\overline{n})-g(\overline{m})\| \leq \beta_1+ \varepsilon'$ for every $(\overline{n},\overline{m}) \in H_1(\M_1')$. \\
We continue inductively this way until we define $\alpha_k$ and $\beta_k$ as follows. \\
Let $\alpha_k=\inf\limits_{ \M_k \in [\M_{k-1}']^{\omega}} \sup\limits_{(\overline{n},\overline{m}) \in H_k(\M_k)} \|f(\overline{n})-f(\overline{m})\|$. \\
There exists $\M_k \in [\M_{k-1}']^{\omega}$ so that $\|f(\overline{n})-f(\overline{m})\| \leq \alpha_k+ \varepsilon'$ for every $(\overline{n},\overline{m}) \in H_k(\M_k)$. \\
Let $\beta_k=\inf\limits_{ \M_k' \in [\M_k]^{\omega}} \sup\limits_{(\overline{n},\overline{m}) \in H_k(\M_k')} \|g(\overline{n})-g(\overline{m})\|$. \\
There exists $\M_k' \in [\M_k]^{\omega}$ so that $\|g(\overline{n})-g(\overline{m})\| \leq \beta_k+ \varepsilon'$ for every $(\overline{n},\overline{m}) \in H_k(\M_k')$. \\
$\bullet$ Let us begin by showing that $\| \alpha_j e_1 + \beta_j e_2 \| \leq \gamma_j$ for all $j \in \{1, \cdots, k\}$. \\
For that, assume that there exists $j \in \{1, \cdots, k\}$ such that $\| \alpha_j e_1 + \beta_j e_2 \| > \gamma_j$. Then, there exists $\eta > 0$ so that $\| (\alpha_j-\eta) e_1 + (\beta_j-\eta) e_2 \| > \gamma_j$. \\
$\ast$ If there exists $(\overline{n}, \overline{m}) \in H_j(\M_k')$ such that $\|f(\overline{n})-f(\overline{m})\| \geq \alpha_j- \eta$ and $\|g(\overline{n})-g(\overline{m})\| \geq \beta_j- \eta$, then $\| h(\overline{n})-h(\overline{m})\| > \gamma_j$, which is impossible. \\
$\ast$ So $\|f(\overline{n})-f(\overline{m})\| \leq \alpha_j- \eta$ or $\|g(\overline{n})-g(\overline{m})\| \leq \beta_j- \eta$ for all $(\overline{n}, \overline{m}) \in H_j(\M_k')$. \\
Now we note that $H_j(\M_k')$ can be identified with $[\M_k']^{k+1}$ so, by Ramsey's Theorem, we get $\M' \in [\M_k']^{\omega}$ such that $\|f(\overline{n})-f(\overline{m})\| \leq \alpha_j- \eta$ for all $(\overline{n}, \overline{m}) \in H_j(\M')$ or $\|g(\overline{n})-g(\overline{m})\| \leq \beta_j- \eta$ for all $(\overline{n}, \overline{m}) \in H_j(\M')$. This contradicts the definition of $\alpha_j$ or $\beta_j$. \\
Thus $\| \alpha_j e_1 + \beta_j e_2 \| \leq \gamma_j$ for all $j \in \{1, \cdots, k\}$. \\
$\bullet$ By assumption, there exists $\M' \in [M_k']^{\omega}$ so that 
\[ \|f(\overline{n})-f(\overline{m}) \| \leq \lambda \left( \sum_{j=1}^k (\alpha_j+\varepsilon')^p \right)^{\frac{1}{p}} \text{ and } \|g(\overline{n})-g(\overline{m}) \| \leq \lambda \left( \sum_{j=1}^k (\beta_j+\varepsilon')^p \right)^{\frac{1}{p}} \]
for all $(\overline{n},\overline{m}) \in I_k(\M')$.
Let $x^n=(\alpha_n+ \varepsilon')e_1 + (\beta_n+\varepsilon')e_2$ for each $n \in \{1, \cdots, k\}$. Using $p$-convexity, we get :
\begin{align*}
\|h(\overline{n})-h(\overline{m})\|^p & \leq \lambda^p \Big\| \left( \sum_{j=1}^k (\alpha_j+\varepsilon')^p \right)^{\frac{1}{p}} e_1 + \left( \sum_{j=1}^k (\beta_j+\varepsilon')^p \right)^{\frac{1}{p}} e_2 \Big\|^p \\
& \leq \lambda^p \sum_{n=1}^k \|x^n\|^p = \lambda^p \sum_{j=1}^k \| (\alpha_j+ \varepsilon')e_1 + (\beta_j+\varepsilon')e_2 \|^p \\
& \leq \lambda^p \sum_{j=1}^k (\gamma_j + 2 \varepsilon')^p 
\end{align*}
for all $(\overline{n}, \overline{m}) \in I_k(\M')$. \\
Therefore,
\[ \|h(\overline{n})-h(\overline{m})\|^p \leq (\lambda+ \varepsilon) \left( \sum\limits_{j=1}^k \gamma_j^p \right)^{\frac{1}{p}}, \] 
for all $(\overline{n}, \overline{m}) \in I_k(\M')$, \ie, $X$ has $(\lambda+\varepsilon)$-HIC$_{p,d}$.
\end{proof}

From this property about finite sums, we can deduce our main result. In order to do so, let us remark that a Banach space $E$ that has a $p$-convex basis with constant $1$ satisfies the following: if $x \in E$ and $(x_n)_{n \in \N}$ is a weakly null sequence in $E$, then
\[ \limsup \|x+x_n\|^p \leq \|x\|^p + \limsup \|x_n \|^p . \]
Therefore, we deduce from the proof of Theorem 4.2 \cite{KR} that if $E$ is in addition reflexive, then for every $k \in \N$, every $\M \in [\N]^{\omega}$, every $\varepsilon > 0$ and every Lipschitz function $f : ([\M]^k, \dH) \to E$, there exist $\M' \in [\M]^{\omega}$ and $u \in E$ so that 
\[ \|f(\overline{n})-u\| \leq \left( \sum_{j=1}^k \alpha_j^p \right)^{\frac{1}{p}}+ \varepsilon \]
for all $\overline{n} \in [\M']^k$, where $\alpha_j=\sup\limits_{(\overline{n},\overline{m}) \in H_j(\M)} \|f(\overline{n})-f(\overline{m})\|$ for all $j \in \{1, \cdots, k\}$.

We now prove Theorem \ref{thm somme E}.

\begin{proof}[Proof of Theorem \ref{thm somme E}]
Let $\varepsilon > 0$, $\M \in [\N]^{\omega}$, $k \in \N$, $f : ([\M]^k,\dH) \to X$ a Lipschitz function. \\
There exists $\varepsilon' > 0$ such that 
\[ (\lambda+2+\varepsilon') \left( \sum\limits_{j=1}^k \alpha_j^p \right)^{\frac{1}{p}} +4 \varepsilon' \leq (\lambda+2+\varepsilon)\left( \sum\limits_{j=1}^k \alpha_j^p \right)^{\frac{1}{p}}\] 
where 
\[ \alpha_j=\sup\limits_{(\overline{n},\overline{m}) \in H_j(\M)} \|f(\overline{n})-f(\overline{m})\| \] for all $j \in \{1, \cdots, k\}$. \\
The well-defined map \[ \phi : \left\lbrace \begin{array}{lll}
X & \to & E \\
(x_n)_{n \in \N} & \mapsto & \sum\limits_{n=1}^{\infty} \|x_n\| e_n
\end{array} \right. \] 
satisfies $\Lip(\phi) \leq 1$ and $\| \phi(x)\|=\|x\|$ for all $x \in X$, thus
\[ \sup\limits_{(\overline{n},\overline{m}) \in H_j(\M)} \| \phi \circ f(\overline{n})-\phi \circ f (\overline{m}) \| \leq \alpha_j \]
for every $j \in \{1, \cdots, k\}$. \\
From the previous remark, we get $u \in E$ and $\M' \in [\M]^{\omega}$ such that 
\[ \| \phi \circ f(\overline{n})-u \| \leq \left( \sum_{j=1}^k \alpha_j^p \right)^{\frac{1}{p}}+ \varepsilon' \]
for all $\overline{n} \in [\M']^k$. Let $N \in \N$ such that $\Big\| \sum_{k=N+1}^{\infty} \|u_k\| e_k \Big\| \leq \varepsilon'$. \\
For each $n \in \N$, let us denote by $P_n$ the projection from $X$ onto $X_n$ and $\Pi_n$ the projection from $X$ onto $\left( \sum_{k=1}^n X_k \right)_E$. \\
We have
\begin{align*}
\Big\| \sum_{n=N+1}^{\infty} \|P_n \circ f(\overline{n})\| e_n \Big\| & \leq \Big\| \sum_{n=N+1}^{\infty} \|P_n \circ f(\overline{n})\| e_n \Big\| - \Big\| \sum\limits_{n=N+1}^{\infty} \|u_n\| e_n \Big\| + \varepsilon' \\
& \leq \Big\| \sum_{n=N+1}^{\infty} (\|P_n \circ f(\overline{n})\| - \|u_n \|) e_n \Big\| + \varepsilon' \\
& \leq \| \phi \circ f(\overline{n})-u \| + \varepsilon' \\
& \leq \left( \sum_{j=1}^k \alpha_j^p \right)^{\frac{1}{p}}+ 2 \varepsilon'
\end{align*}
for all $\overline{n} \in [\M']^k$. \\
Moreover, according to Proposition \ref{somme 2 esp}, we get an infinite subset $\M'' \in [\M']^{\omega}$ such that 
\[ \| \Pi_N \circ f(\overline{n})-\Pi_N \circ f(\overline{m}) \| \leq (\lambda+\varepsilon') \left( \sum_{j=1}^k \alpha_j^p \right)^{\frac{1}{p}} \]
for all $(\overline{n}, \overline{m}) \in I_k(\M'')$. \\
We deduce
\begin{align*}
\|f(\overline{n})-f(\overline{m})\| & \leq \| \Pi_N(f(\overline{n})-f(\overline{m})) \| + \| (I-\Pi_N) \circ f(\overline{n})\|  + \| (I-\Pi_N) \circ f(\overline{m}) \| \\
& \leq (\lambda+\varepsilon') \left( \sum_{j=1}^k \alpha_j^p \right)^{\frac{1}{p}} + \left( \sum_{j=1}^k \alpha_j^p \right)^{\frac{1}{p}}+ 2 \varepsilon' + \left( \sum_{j=1}^k \alpha_j^p \right)^{\frac{1}{p}}+ 2 \varepsilon' \\
& \leq (\lambda+2+\varepsilon)\left( \sum\limits_{j=1}^k \alpha_j^p \right)^{\frac{1}{p}}
\end{align*}
for all $(\overline{n}, \overline{m}) \in I_k(\M'')$. The result follows.
\end{proof}

\begin{rmk}
With this result and Proposition \ref{prop KRI=K}, we immediately deduce the following: if each $X_n$, $n \in \N$, has property $\lambda$-HC$_{p,d}$, then $(\sum_{n \in \N} X_n)_E$ has property $(2 \lambda + 2 + \varepsilon)$-HC$_{p,d}$ for every $\varepsilon > 0$.
\end{rmk}

Once again, we chose to state this theorem with property HIC$_{p,d}$, but the result stays true for property HFC$_{p,d}$, with a similar proof.

\begin{rmk}
$\bullet$ Of course, the condition that all spaces have property HIC$_{p,d}$ with the same constant is essential because 
\[ ([\N]^{< \omega}, \dH) \underset{L}{\hookrightarrow} X_{\omega}=\left( \sum\limits_{n=1}^{\infty} \ell_1^n(\ell_2) \right)_{\ell_2} \] 
even though $\ell_1^n(\ell_2)$ has property HFC$_{2,d}$ (it is reflexive and $2$-AUS) for every $n \in \N$. \\
To see that, let us note that, for every $k \in \N$, $([\N]^{ \leq k}, \dH)$ isometrically embeds into $\ell_1^k(\ell_2)$. Then, the barycentric gluing technique by Baudier (see \cite{baudier2020barycentric}) gives us a bi-Lipschitz embedding from $([\N]^{< \omega}, \dH)$ into $X_{\omega}$. \\
$\bullet$ In \cite{Braga1}, Braga asked the following (Problem 3.7): if a Banach space $X$ has the Banach-Saks property, \ie, every bounded sequence in $X$ admits a subsequence whose Ces\`aro means converge in norm, does it follow that $([\N]^{< \omega}, d_{\Delta})$ does not Lipschitz embed into $X$ ? The answer to this question is negative. Indeed, let $(p_n)_{n \in \N} \subset (1, \infty)$ be a decreasing sequence such that $\liminfty p_n=1$ and let $X=\left( \sum_{n=1}^{\infty} \ell_{p_n} \right)_{\ell_2}$. With a similar argument, or simply appealing to Ribe's Theorem \cite{ribe} (that implies that $\ell_1$ coarse Lipschitz embeds into $X$), we see that 
\[ ([\N]^{< \omega}, \dH) \underset{L}{\hookrightarrow} X \text{ and } ([\N]^{< \omega}, d_{\Delta}) \underset{L}{\hookrightarrow} X \]
even though the space $X$ has the Banach-Saks property (see \cite{partington}).
\end{rmk}

Before we write a direct consequence of this theorem, let us briefly recall the definition of the James sequence spaces. \\
Let $p \in (1, \infty )$. The James space $\J_p$ is the real Banach space of all sequences $x=(x(n))_{n \in \N}$ of real numbers with finite $p$-variation and verifying $\lim_{n \to \infty} x(n)=0$. The space $\J_p$ is endowed with the following norm
\[ \|x\|_{\J_p}= \sup \left\{ \left( \sum_{i=1}^{k-1} |x(p_{i+1})-x(p_i)|^p \right)^{\frac{1}{p}} ; 1 \leq p_1 < p_2 < \cdots < p_k \right\} . \]
The space $\J=\J_2$, constructed by James in \cite{james}, is the historical example of a quasi-reflexive Banach space which is isomorphic to its bidual. In fact, $\J_p^{**}$ can be seen as the space of all sequences of real numbers with finite $p$-variation, which is $\J_p \oplus \R e$, where $e$ denotes the constant sequence equal to $1$. \\
Besides of being quasi-reflexive, the space $\J_p$ has the property of being $p$-AUSable (see \cite{Netillard}, Proposition 2.3) and its dual $\J_p^*$ is $q$-AUSable, where $q$ denotes the conjugate exponent of $p$ (see \cite{LPP} and references therein).

We can now state the following corollary.

\begin{corollary} \label{cor thm somme E}
Let $p,q \in (1, \infty)$, $p'$ the conjugate exponent of $p$, $s= \min(p,q)$ and $t=\min(p',q)$. \\
If $X$ a quasi-reflexive Banach space satisfying upper $\ell_p$ tree estimates, then the space $\ell_q(X)$ has property HC$_{s,d}$. \\
In particular, $\ell_q(\J_p)$ has property HC$_{s,d}$ and $\ell_q(\J_p^*)$ has property HC$_{t,d}$.
\end{corollary}

Let us mention that we stated this corollary for $\ell_q$-sums but we could have done it with any reflexive $q$-convexification of a Banach space with a $1$-unconditional basis (such as $T_q$, the $q$-convexification of Tsirelson space, or $S_q$, the $q$-convexification of Schlumprecht space, see \cite{Braga1} and references therein).

With $p=2$, we get that the spaces $\ell_2(\J)$ and $\ell_2(\J^*)$ have property HC$_2$ and thus  cannot contain equi-Lipschitz copies of Hamming graphs. In fact, property HC$_p$ provides more information than an obstruction to the equi-Lipschitz embedding of Hamming graphs, it also gives us an estimation of some compression exponents, given by the result below. Before stating it, we need the following definition.

\begin{definition}
Let $q \in (1, \infty)$ and $X$ be a Banach space. We say that $X$ has the \textit{$q$-co-Banach-Saks property} if for every semi-normalized weakly null sequence $(x_n)_{n \in \N}$ in $X$, there exists a subsequence $(x'_n)_{n \in \N}$ of $(x_n)_{n \in \N}$ and $c > 0$ such that, for all $k \in \N$ and all $k \leq n_1 < \cdots < n_k$, we have
\[ \|x'_{n_1} + \cdots + x'_{n_k} \| \geq c k^{1/q} . \]
\end{definition}

\begin{theorem}
Let $1 < q < p$ in $(1, \infty)$. Assume $X$ is an infinite-dimensional Banach space with the $q$-co-Banach-Saks property and $Y$ is a Banach space with property HC$_p$. Then $X$ does not coarse Lipschitz embed into $Y$. More precisely, the compression exponent $\alpha_Y(X)$ of $X$ into $Y$ satisfies the following: \\
$(i)$ if $X$ contains an isomorphic copy of $\ell_1$, then $\alpha_Y(X) \leq \frac{1}{p}$; \\
$(ii)$ otherwise, $\alpha_Y(X) \leq \frac{q}{p}$. \\
In particular, if $X$ is $q$-AUC, then $\alpha_Y(X) \leq \frac{q}{p}$.
\end{theorem}

We refer the reader to Theorem 3.5 and Corollary 3.6 of \cite{LR} for a proof of this result. 

Let us note that Proposition 3.2 of \cite{LR} also stays true by replacing ``quasi-reflexive AUS" by ``having property HC$_p$ for some $p \in (1, \infty)$". \\ 

We also would like to mention the following: we could define symmetric concentration properties \textit{SFC$_p$}, \textit{SIP$_p$} and \textit{SC$_p$}, corresponding respectively to properties HFC$_p$, HIC$_p$ and HC$_p$ by asking the function $f$ to be Lipschitz for the symmetric distance in the definitions of these properties (instead of being Lipschitz for the Hamming distance).
Then, it is known that a reflexive (resp. quasi-reflexive) $p$-AUS Banach space, for $p \in (1, \infty)$, would have property SFC$_p$ (resp. SC$_p$). Moreover, even though we wrote our properties HFC$_{p,d}$, HIC$_{p,d}$ and HC$_{p,d}$ with the letter ``H" because the quantities $\alpha_j$, $j \in \{1, \cdots, k \}$ can be seen as directional Lipschitz constants when $[\N]^k$ is endowed with the Hamming distance, we could replace ``$f : ([\N]^k, \dH) \to X$ Lipschitz" by ``$f : [\N]^k \to X$ bounded" in the definitions so that no reference to any specific metric is made. With that remark in mind and the fact that HC$_{p,d}$ implies SC$_p$, we get that property HC$_{p,d}$ prevents the equi-Lipschitz embeddings of the symmetric graphs.

Before concluding this subsection with a last result, let us recall some facts that we will use concerning the spaces $\operatorname{Ti}_q^*$, the dual of the $q$-convexification of the Tirilman space $\operatorname{Ti}$ (see \cite{tzafriri}, \cite{CS}, \cite{3.5} and references therein for more information about this space), and $S_q^*$, the dual of the $q$-convexification of Schlumprecht space $S$ (see \cite{schlumprecht_space}, \cite{3.5} and references therein for more information about this space). \\
If we denote  $(e_n^*)_{n \in \N}$ the coordinate functionals associated with the canonical basis $(e_n)_{n \in \N}$ of $\operatorname{Ti}$, it is known that $(e_n^*)_{n \in \N}$ is $1$-symmetric and that the formal identity $I : \ell_q \to \operatorname{Ti}_q$ is bounded and stricty singular (see \cite{3.5}). As for $S_q^*$, if we denote $(f_n^*)_{n \in \N}$ the sequence of coordinate functionals associated with the canonical basis $(f_n)_{n \in \N}$ of $S$ and if $p$ and $q$ are conjugate exponents, it is known (see \cite[Proposition 6.5 (iv)]{3.5}) that for any finite, non-empty subset $E$ of $\N$, 
\[ \Big\| \sum_{i \in E} f_i^* \Big\|_{S_q^*} \geq |E|^{1/p} \log_2(|E|+1)^{1/q} , \]
and that $(f_n^*)_{n \in \N}$ is $1$-subsymmetric.

\begin{prop}
Let $p \in (1, \infty)$, $q$ its conjugate exponent. \\
The space $\operatorname{Ti}_q^*$ has property HFC$_p$ but does not have property HC$_{p,d}$, and the space $S_q^*$ has property HFC$_{s,d}$ for every $s \in (1,p)$, but does not have property HC$_p$.
\end{prop}

\begin{proof}
This space $\operatorname{Ti}_q^*$ is reflexive and $\operatorname{Ti}_q^* \in \Np$ (see \cite{3.5} Proposition 6.5 (v)) hence it has property HFC$_p$. Now, for $a=(a_j)_{j=1}^k \in B_{\ell_p^k}$, we let $f : \left\lbrace \begin{array}{lll}
[\N]^k & \to & \operatorname{Ti}_q^* \\
\overline{n} & \mapsto & \sum_{j=1}^k a_j e_{n_j}^*
\end{array} \right.$. This map satisfies $\Lip(f) \leq 2$. \\
Let us assume $\operatorname{Ti}_q^*$ has property $\lambda$-HC$_{p,d}$ for some $\lambda>0$. Then, there exist $\overline{n}, \overline{m} \in [\N]^k$ such that $\nb \cap \mb = \varnothing$ and
\[ \Big\| \sum_{j=1}^k a_j e_j^* \Big\| = \|f(\overline{n})\| \leq \|f(\overline{n})-f(\overline{m})\| \leq 2 \lambda \]
because of the $1$-symmetry of $(e_j^*)_{j \in \N}$. We deduce that the sequence of coordinate functionals $(e_j^*)_{j \in \N}$ is dominated by the $\ell_p$ basis. This is impossible, by the same argument used by Causey \cite{3.5} to prove $\operatorname{Ti}_q^* \notin \Ap$ (by duality, the $\operatorname{Ti}_q$ basis would dominate, and would therefore be equivalent to the $\ell_q$ basis, contradicting the strict singularity of the formal inclusion $I : \ell_q \to \operatorname{Ti}_q$, cf. \cite[Proposition 6.5 (ii)]{3.5}). \\ 
Then, the space $S_q^*$ is reflexive and $S_q^* \in \Pp$ (see the proof of Theorems 6.2, 6.3, case $\xi=0$, and Remark 6.7 in \cite{3.5}) hence it has property HFC$_{s,d}$ for every $s \in (1,p)$. We define similarly, for $a=(a_j)_{j=1}^k \in B_{\ell_p^k}$, a $2$-Lipschitz map $f : \left\lbrace \begin{array}{lll}
[\N]^k & \to & S_q^* \\
\overline{n} & \mapsto & \sum_{j=1}^k a_j f_{n_j}^*
\end{array} \right.$. Now, we may argue as we did for $\operatorname{Ti}_q^*$ to deduce that, for all $(\nb,\mb) \in I_k(\N)$:
\[ \|f(\nb)-f(\mb)\| \geq k^{1/p} \log_2(k+1)^{1/q} \]
because of the $1$-subsymmetry of the canonical basis and \cite[Proposition 6.5 (iv)]{3.5}. Since $\lim_{k \to \infty} \log_2(k+1)^{1/q} = \infty$, the space $S_q^*$ cannot have property HC$_p$.
\end{proof}

\subsection{Related questions}

The following questions about Theorem \ref{thm somme E} come up naturally.

\begin{pb}
Can we replace property HFC$_{p,d}$ (resp. HIC$_{p,d}$) by property HFC$_p$ (resp. HIC$_p$) in Theorem \ref{thm somme E}?
\end{pb}

\begin{pb}
Can the conclusion of Theorem \ref{thm somme E} be improved so that $X= \left( \sum_{n \in \N} X_n \right)_E$ has property $(\lambda+\varepsilon)$-HIC$_{p,d}$ for every $\varepsilon > 0$?
\end{pb}

\begin{pb}
Let $p \in (1, \infty )$, $X$ a $p$-AUS Banach space so that $X$ is complemented in $X^{**}$ and that $X^{**} / X$ is reflexive and $p$-AUSable. Does $X$ have property HC$_p$?
\end{pb}

A positive answer to the second question would provide us, for each $p \in (1, \infty )$, with an example of a reflexive Banach space, not AUSable, with property HFC$_{p,d}$. Indeed, following Braga's proof of Theorem 7.1 \cite{Braga2}, the space $X_{p, \ell_1, T}$ would be such an example (see \cite{Braga2} and references therein for more information about this space). \\

Moreover, let us recall that Kalton proved the existence of a Banach space $X$ that is not $p$-AUSable but that is uniformly homeomorphic to a $p$-AUS Banach space (see \cite{kalton2013examples}). Thus, the space $X$ has property HFC$_{p,d}$, $p \in (1, \infty)$, even though it is not $p$-AUSable. However, the following problem remains open.

\begin{pb}
Is there a Banach space that has property HFC$_p$ (or HFC$_{p,d}$/HC$_p$/HC$_{p,d}$) without being AUSable? If a Banach space $X$ coarse Lipschitz embeds into a Banach space $Y$ that is reflexive and AUS, does it follow that $X$ is AUSable?
\end{pb}

We will finish this section by saying a few words about a natural class of spaces to study here: the Lindenstrauss spaces (see \cite{espLin}). For any Banach space $X$, we will note by $Z_X$ the Lindenstrauss space associated to $X$. In \cite{Braga2}, Braga showed that neither $Z_{c_0}^*$, $Z_{\ell_1}$ or $Z_{X_{\omega}^*}^*$ can have any of the concentration properties we introduced, even though they are $2$-AUSable (see \cite{CL}) and do not contain $c_0$ nor $\ell_1$. The key point of the proof for the spaces $Z_{c_0}^*$ and $Z_{X_{\omega}^*}^*$ is that they satisfy the assumptions of the following proposition, that can be deduced from \cite{Braga2}.

\begin{prop} \label{prop bidual Braga}
Let $X$ be a Banach space such that $X^*$ is separable. \\
Assume that there exist $A,C \geq 1$, $(z_{k,j,n}^{**})_{k \in \N, j \in \{1,\cdots,k\}, n \in \N} \subset C B_{X^{**}}$ such that for every $k \in \N$, the map \[ F_k : \left\lbrace \begin{array}{lll}
[\N]^k & \to & X^{**} \\
\overline{n} & \mapsto & \sum\limits_{j=1}^k z^{**}_{k,j,n_j}
\end{array}
\right.\] 
satisfies
\[ \dfrac{1}{A} \dH(\overline{n},\overline{m}) \leq \|F_k(\overline{n})-F_k(\overline{m})\| \leq A \dH(\overline{n},\overline{m}) \]
for all $\overline{n},\overline{m} \in [\N]^k$. \\
Then, the space X does not have any of the concentration properties introduced before.
\end{prop}

We can therefore ask ourselves the following question.

\begin{pb}
Can we find an infinite-dimensional Banach space $X$ and a $p \in (1, \infty )$ such that $Z_X$ or $Z_X^*$ have property HC$_p$?
\end{pb}

Finally, by Aharoni's Theorem \cite{Aharoni}, we know that the Hamming graphs equi-Lipschitz embed into $Z_{c_0}^{**}/Z_{c_0}$. Does it mean that these graphs can be Lipschitz embedded into $Z_{c_0}^{**}$? Into $Z_{c_0}$? 

\section{Asymptotic-$c_0$ spaces} \label{section asympt}

Before stating the last result of this paper, we recall the definition of an asymptotic-$c_0$ space. The following definition is due to Maurey, Milman and Tomczak-Jaegermann \cite{maurey1995asymptotic}.

\begin{definition}
Let $X$ be a Banach space. We denote by $\cof(X)$ the set of all its closed finite-codimensional subspaces. \\
For $C \geq 1$, we say that $X$ is \textit{$C$-asymptotically $c_0$} if, for any $k \in \N$, we have 
\begin{align*}
& \exists X_1 \in \cof(X) \hspace{1mm} \forall x_1 \in S_{X_1} \hspace{1mm} \exists X_2 \in \cof(X) \hspace{1mm} \forall x_2 \in S_{X_2} \hspace{1mm} \cdots \hspace{1mm} \exists X_k \in \cof(X) \hspace{1mm} \forall x_k \in S_{X_k}, \\ 
& \hspace{1.8cm} \forall (a_1, \dots, a_k) \in \R^k,  \Big\| \sum_{i=1}^k a_i x_i \Big\| \leq C \max_{1 \leq i \leq k} |a_i |
\end{align*}
We say that $X$ is \textit{asymptotically $c_0$} (or asymptotic-$c_0$) if it is $C$-asymptotically $c_0$ for some $C \geq 1$.
\end{definition}

Let $X$ be a Banach space. A family $(x_j^{(i)}; i,j \in \N) \subset X$ is called an \textit{infinite array}. For an infinite array $(x_j^{(i)}; i,j \in \N)$, we call the sequence $(x_j^{(i)})_{j \in \N}$ the \textit{$i$-th row of the array}. We call an array \textit{weakly null} if all rows are weakly null. A \textit{subarray} of $(x_j^{(i)} ; i,j \in \N)$ is an infinite array of the form $(x_{j_s}^{(i)}; i,s \in \N)$, where $(j_s) \subset \N$ is a subsequence. Thus, for a subarray, we are taking the same subsequence in each row. 

The following notion, introduced by Halbeisen and Odell (\cite{MR2041794}), is a generalization of spreading models.

\begin{definition}
A basic sequence $(e_i)_{i \in \N}$ is called an \textit{asymptotic model} of a Banach space $X$ if there exist an infinite array $(x_j^{(i)} ; i,j \in \N) \subset S_X$ and a null-sequence $(\varepsilon_n)_{n \in \N} \subset (0,1)$, so that, for all $n \in \N$, all $(a_i)_{i=1}^n \subset [-1,1]$ and all $n \leq k_1 < k_2 < \cdots < k_n$,
\[ \left| \| \sum_{i=1}^n a_i x^{(i)}_{k_i} \| - \| \sum_{i=1}^n a_i e_i \| \right| < \varepsilon_n . \]
\end{definition}

The following proposition concerning this notion was proved in \cite{MR2041794}.

\begin{prop}[\cite{MR2041794}, Proposition 4.1 and Remark 4.7.5]
Assume that $(x_j^{(i)} ; i,j \in \N )\subset S_X$ is an infinite array, all of whose rows are normalized and weakly null. Then there is a subarray of $(x_j^{(i)} ; i,j \in \N)$ which has a $1$-suppression-unconditional asymptotic model $(e_i)_{i \in \N}$.
\end{prop}

We call a basic sequence $(e_i)_{i \in \N}$ \textit{$c$-suppression-unconditional}, for some $c \geq 1$, if, for any $(a_i)_{i \in \N} \subset c_{00}$ and any $A \subset \N$, we have :
\[ \Big\| \sum_{i \in A} a_i e_i \Big\| \leq c \Big\| \sum_{i=1}^{\infty} a_i e_i \Big\| . \]
Note that a $c$-unconditional basic sequence is $c$-suppression-unconditional and a $c$-suppression-unconditional basic sequence is $2c$-unconditional.

As for the proof of the fact that every Banach space with property HFC$_{\infty}$ is asymptotic-$c_0$ (see \cite{Baudier2020ANC}), the key ingredient will be the following theorem of Freeman, Odell, Sari and Zheng.

\begin{theorem} [\cite{MR3841837}, Theorem 4.6] \label{key theorem}
If a separable Banach space $X$ does not contain any isomorphic copy of $\ell_1$ and all the asymptotic models generated by normalized weakly null arrays are equivalent to the $c_0$ unit vector basis, then $X$ is asymptotically $c_0$.
\end{theorem}

We now have all the tools to prove our result.

\begin{theorem} \label{HC infty implique asc0}
If a Banach space has property HC$_{\infty}$, then it is asymptotic-$c_0$.
\end{theorem}

\begin{proof}
Let $X$ be a Banach space with property HC$_{\infty}$. Then $X$ has property $\lambda$-HIC$_{\infty}$, for some $\lambda > 0$, by Proposition \ref{prop KRI=K}. Let us note that we can assume that $X$ is separable by Proposition 11 of \cite{DK2}, and that $X$ cannot contain an isomorphic copy of $\ell_1$ since $\ell_1$ does not have this property. \\
Assume by contradiction that $X$ is not asymptotic-$c_0$. By Theorem \ref{key theorem}, there exists a $1$-suppression-unconditional sequence $(e_i)_{i \in \N}$ that is not equivalent to the unit vector basis of $c_0$, and hence $\lambda_k=  \Big\| \sum_{i=1}^{k} (-1)^i e_i \Big\| \nearrow \infty$, if $k \nearrow \infty$, and that is generated as an asymptotic model of a normalized weakly null array $(x_j^{(i)} ; i,j \in \N)$ in $X$. Let $k \in \N$ such that $\dfrac{\lambda_{2k}}{4} > \lambda$, and $\delta=\dfrac{\lambda_{2k}}{2}$. After passing to appropriate subsequences of the array, we may assume that, for any $2k \leq j_1 < \cdots < j_{2k}$ and any $a_1, \dots, a_{2k} \in [-1,1]$, we have
\begin{align}
\left| \Big\| \sum_{i=1}^{2k} a_i x_{j_i}^{(i)} \Big\|- \Big\| \sum_{i=1}^{2k} a_i e_i \Big\| \right| < \delta .
\end{align}
Define now $f(\overline{m})=\frac{1}{2} \sum_{i=1}^k x_{m_i}^{(i)}$ for $\overline{m}=(m_1, \dots, m_{k}) \in [\N]^{k}$. Note that $f$ is $1$-Lipschitz for the metric $d_{\mathbb{H}}$. \\
Let $\M \in [\N]^{\omega}$, $\overline{n}, \overline{m} \in [\M]^k$ such that $((n_1, \dots, n_k), (m_1, \dots, m_k)) \in I_k(\M)$ and $m_1, n_1 > 2k$. \\
Using the Hahn-Banach Theorem, we can find $x^* \in S_{X^*}$ such that :
\[ x^* \left( \sum\limits_{i=1}^k x_{m_i}^{(i)} - \sum\limits_{i=1}^k x_{n_i}^{(i)} \right) = \Big\| \sum\limits_{i=1}^{k} x_{m_i}^{(i)} - \sum\limits_{i=1}^{k} x_{n_i}^{(i)} \Big\| . \]
Using equation $(1)$, we deduce that :
\[ \|f(m_1, \dots, m_k)-f(n_1, \dots, n_k) \| \geq \frac{1}{2} \lambda_{2k}-\frac{\delta}{2} = \frac{\lambda_{2k}}{4} > \lambda . \]
This contradicts the assumption on $X$. The result follows.
\end{proof}

Theorem \ref{thm somme E} and Corollary \ref{cor thm somme E} provided us with examples of non-quasi-reflexive Banach spaces having property HC$_p$, for $p \in (1, \infty)$. In order to obtain a similar result with $p= \infty$, it seems natural to consider a $T^*$-sum of spaces with $\lambda$-HFC$_{\infty}$, $\lambda > 0$. However, as a direct consequence of Theorem \ref{HC infty implique asc0}, we get the following corollary.

\begin{corollary} \label{cor sum T*}
The space $T^*(T^*)$, where $T^*$ is the original Banach space constructed by Tsirelson in \cite{Tsirelson}, does not have property HC$_{\infty}$ although it has property HFC$_{p,d}$ for every $p \in (1, \infty)$.
\end{corollary}

Before proving this corollary, let us recall some properties of the space $T^*$. First, it is reflexive, so we will denote its dual by $T$. This space $T^*$ has a normalized, shrinking, $1$-unconditional basis $(e_n)_{n=1}^{\infty}$. Let us denote by $(e_n^*)_{n \in \N} \subset T$ the coordinate functionals of $(e_n)_{n \in \N}$. For an element $x= \sum_{n=1}^{\infty} e_n^*(x) e_n \in X$, let us denote by $\supp(x)$ the support of $x$, \ie, the subset of integers $n$ such that $e_n^*(x) \neq 0$. The space $T$ satisfies the following (see \cite{Tsirelson} and \cite{FJ}): for every $(x_i)_{i=1}^n \subset T$ with $(\supp (x_i))_{i=1}^n$ increasing, $\|x_i\|=1$, and $\supp(x_1) \subset [k+1, \infty )$, we have
\[ \forall (a_i)_{i=1}^n \subset \R, \Big\| \sum_{i=1}^n a_i x_i \Big\| \geq \frac{1}{2} \sum_{i=1}^n |a_i| . \]  

\begin{proof}[Proof of Corollary \ref{cor sum T*}]
The lack of property HC$_{\infty}$ is a direct consequence of Lemma 2.7 of \cite{baudier2020geometry} (asserting that $T^*(T^*)$ is not asymptotic-$c_0$) and Theorem \ref{HC infty implique asc0} above. The fact that $T^*(T^*)$ has property HFC$_{p,d}$, for every $p \in (1, \infty)$, can be deduced from Theorem 5.9 of \cite{draga2016direct}, applied with $p=1$, $X_n=T^*$ for all $n \in \N$ and $\mathfrak{E}=T^*$. Indeed, if we let
\[ \operatorname{p}(X)=\inf \left\{ q \geq 1; \text{ $X$ is $p$-AUSable}, \frac{1}{p}+\frac{1}{q}=1 \right\} \]
for a given Banach space $X$, Theorem 5.9 \cite{draga2016direct} asserts that $\operatorname{p}(T^*(T^*))=1$. As $T^*(T^*)$ is reflexive, the result follows from Theorem 6.1 \cite{KR} of Kalton and Randrianarivony (see the remark after Definition 2.5).
\end{proof}

As property HC$_{\infty}$ prevents the equi-coarse embeddability of the Hamming graphs, we can therefore ask the following :

\begin{pb}
If the Hamming graphs do not equi-coarsely embed into a Banach space $X$, does it follow that $X$ is asymptotic-$c_0$?
\end{pb}

Let us mention that, in case the answer to this question is positive, the embeddings in a non-asymptotic-$c_0$ space would not be canonical, because of Proposition 4.6 and Remark 4.7 of \cite{baudier2020geometry}. \\

Moreover, let us note that we only defined property HFC$_{p,d}$ (resp. HIC$_{p,d}$, resp. HC$_{p,d}$) for $p \in (1, \infty)$ because, for $p= \infty$, this is exactly the definition of property HFC$_{\infty}$ (resp. HIC$_{\infty}$, resp. HC$_{\infty}$). In the light of Section \ref{section sums}, the following question seems natural.

\begin{pb}
Does property HC$_{\infty}$ imply quasi-reflexivity?
\end{pb}

In addition, as Lindenstrauss spaces provided us with non trivial examples of AUSable dual Banach spaces without any concentration property, the following result, due to Schlumprecht, provides us with a non trivial example of an asymptotic-$c_0$ separable dual Banach space without any concentration property.

\begin{theorem}[Schlumprecht] \label{thm schlumprecht}
Let $X$ be a Banach space whose dual is separable. Then, there exists an asymptotic-$c_0$ separable dual Banach space $Z_X$ such that
\[ Z_X^{**}=Z_X \oplus X^* . \]
\end{theorem}

With this theorem, proved in Section 5, and arguments of Braga \cite{Braga2}, we can prove the following result, which proof is in the spirit of Proposition \ref{prop bidual Braga}.

\begin{corollary}
There exists a separable asymptotic-$c_0$ dual space $Z$ such that there exists a sequence of equi-Lipschitz functions $(f_k : [\N]^{2k} \to Z)_{k \in \N}$ satisfying the following property: for all $\varepsilon > 0$, all $k \in \N$ and all $\overline{n}, \overline{m} \in [\N]^k$, there exists $i \geq \max(n_k, m_k)$ such that
\[ \|f_k(\overline{n}, \overline{n}')-f_k(\overline{m}, \overline{m}')\| \geq 2 d_{\Delta}(\overline{n}, \overline{m})-\varepsilon \]
for all $\overline{n}', \overline{m}' \in [\N]^k$ with $n_1', m_1' > i$. \\
In particular, $Z$ cannot have any of the concentration properties we introduced. 
\end{corollary}

\begin{proof}
Let $Z=Z_{c_0}$ given by the previous theorem. It is a separable asymptotic-$c_0$ dual space. Now, we start by noting that $\ell_1$ linearly embeds into $Z^{**}$ hence the existence of a bounded sequence $(z_n^{**})_{n \in \N} \subset Z^{**}$ with the following property
\[ (*) \hspace{0.5cm} \forall k \in \N, \forall (\varepsilon_1, \dots, \varepsilon_k) \in \{\pm 1 \}^k, \forall (n_1, \dots, n_k) \in [\N]^k, \Big\| \sum_{j=1}^k \varepsilon_j z_{n_j}^{**} \Big\| \geq k . \]
Let $C= \sup_{n \in \N} \|z_n^{**}\|_{Z^{**}}$. Since $Z^*$ is separable, by Goldstine's Theorem, for each $n \in \N$, we can find a sequence $(z_{(n,m)})_{m \in \N} \subset C B_Z$ such that
\[ z_n^{**} = \omega^* - \lim\limits_{m \to \infty} z_{(n,m)} . \]
Then, for each $k \in \N$, the map
\[ f_k : \left\lbrace \begin{array}{lll}
[\N]^{2k} & \to & Z \\
\overline{n} & \mapsto & \sum\limits_{j=1}^k z_{(n_j,n_{k+j})}
\end{array} \right. \]
satisfies $\Lip(f_k) \leq 2C$. Using weak$^*$-lower semicontinuity of the norm and $(*)$, we get the result. \\
In particular, this sequence of equi-Lipschitz functions is such that 
\[ \forall k \in \N, \forall \M \in [\N]^{\omega}, \exists (\overline{n}, \overline{m}) \in I_{2k}(\M) ; \|f_k(\overline{n})-f_k(\overline{m})\| \geq 2k-1 \]
thus $Z$ cannot have any of the concentration properties we introduced.
\end{proof}

\begin{rmk}
Let us mention that $Z$ is a non-quasi-reflexive asymptotic-$c_0$ space that does not satisfy any of the concentration properties for non-trivial reasons. Indeed, since $Z$ has the Radon-Nikod\`ym property (see \cite{DP}) and a separable bidual, it cannot contain a linear copy of $c_0$, not even a Lipschitz copy of $\ell_1$. 
\end{rmk}

\section{Schlumprecht's construction of generalized Lindenstrauss spaces} \label{app}

This last section is dedicated to Schlumprecht's proof of Theorem \ref{thm schlumprecht}.

Let $X$ be a separable Banach space, and $(x_n)_{n \in \N}$ be a sequence in $S_X$ which has the property that $\{ \pm x_n, n \in \N \}$ is dense in $S_X$. Secondly, we are given a space $U$ which has a normalized and $1$-unconditional basis $(u_n)_{n \in \N}$. 

On $c_{00}$, we define the following norm: for $a= \sum_{i \in \N} a_i e_i \in c_{00}$, put
\[ \|a\| = \max \left\{ \Big\| \sum_{j=1}^k \| \sum_{i \in I_j} a_i x_i \|_X u_{\min I_j} \Big\|_U \big| \begin{matrix} k \in \N, I_j \subset \N \text{ interval}, j \leq k \\ I_1 < I_2 < \cdots < I_k \end{matrix} \right\} . \]
Here $I < J$ for two subsets $I$ and $J$ of $\N$, with $I$ being finite, means $\max I < \min J$, and we write $I \geq n$, or $I \leq n$ if $\min I \geq n$, or $\max I \leq n$.

We let $Z$ be the completion of $c_{00}$ under above defined norm and denote that norm by $\|\cdot\|_Z$.

\begin{rmk}
In \cite{espLin}, the norm $\|\cdot\|_Z$ was defined by taking only intervals $I_1 < I_2 < \cdots < I_k$, with the property that $\min(I_{j+1})=\max(I_j)+1$, if $1 \leq j < k$. But since $(u_j)_{j \in \N}$ is $1$-suppression-unconditional, this is irrelevant. Indeed, if $I_j=[p_j, q_j)$, for $1 \leq j \leq k$, one may add, if necessary, the intervals $[q_j, p_{j+1})$, for $1 \leq j < k-1$, to satisfy the condition in \cite{espLin}.
\end{rmk}

Let us first note some easy observations.

\begin{lemma} \label{lemma 3}
Assume that $a= \sum_{i \in \N} a_i e_i \in c_{00}$ with $m= \min \supp(a)$ and $n= \max \supp(a)$. \\
We have
\[ \|a\|_Z \leq 2 \max \left\{ \Big\| \sum_{j=1}^k \| \sum_{i \in I_j} a_i x_i \|_X u_{\min I_j} \Big\|_U \big| \begin{matrix} k \in \N, I_j \subset \N \text{ interval}, j \leq k \\ m \leq I_1 < \cdots < I_k \leq n \end{matrix}  \right\} . \]
\end{lemma}

\begin{proof}
There exist $k \in \N$ and intervals $I_1 < \cdots < I_k$ so that 
\[ \|a\|_Z = \Big\| \sum_{j=1}^k \| \sum_{i \in I_j} a_i x_i \|_X u_{\min I_j} \Big\|_U . \]
We can assume that $\max I_k \leq n$ (otherwise, $I_k$ would not contribute to $\|a\|_Z$) and we could replace $I_k$ by $I_k \cap [1,n]$ without changing the left side of above equation. \\
Secondly, we can assume that $k \geq 2$ (otherwise the claim is trivial) and that $m \leq I_2$ (otherwise, the contribution of the $I_1$ part is $0$). \\
$\ast$ If $\| \sum_{i \in I_1} a_i x_i \|_X \leq \frac{1}{2} \|a\|_Z$, then
\begin{align*}
\|a\|_Z & \leq 2 \Big\| \sum_{j=2}^k \| \sum_{i \in I_j} a_i x_i \|_X u_{\min I_j} \Big\|_U \\
& \leq 2 \max \left\{ \Big\| \sum_{j=1}^k \| \sum_{i \in J_j} a_i x_i \|_X u_{\min J_j} \Big\|_U ; k \in \N, J_j \subset \N \text{ interval}, \begin{matrix} j \leq k \\ m \leq J_1 <  \cdots < J_k \leq n \end{matrix}  \right\} .
\end{align*}
$\ast$ Otherwise, $\| \sum_{i \in I_1} a_i x_i \|_X > \frac{1}{2} \|a\|_Z$ and
\begin{align*}
\|a\|_Z & \leq 2 \| \sum_{i \in I_1} a_i x_i \|_X  = 2 \Big\| \| \sum_{i \in I_1 \cap [m,n]} a_i x_i \|_X u_m \Big\|_U \\
& \leq 2 \max \left\{ \Big\| \sum_{j=1}^k \| \sum_{i \in J_j} a_i x_i \|_X u_{\min J_j} \Big\|_U \big| \begin{matrix} k \in \N, J_j \subset \N \text{ interval}, j \leq k \\ m \leq J_1 <  \cdots < J_k \leq n \end{matrix}  \right\}
\end{align*}
which proves the claim.
\end{proof}

\begin{prop} \label{prop 2}
We have \\
$(i)$ $(e_j)_{j \in \N}$ is a normalized monotone basis of $Z$; \\
$(ii)$ If $(u_j)_{j \in \N}$ is boundedly complete so is $(e_j)_{j \in \N}$; \\
$(iii)$ For $\sum_{i \in \N} a_i e_i \in Z$ it follows that $\sum_{i \in \N} a_i x_i$ converges in $X$, and the map
\[ Q : \left\lbrace \begin{array}{lll}
Z & \to & X \\
\sum_{j \in \N} a_j e_j & \mapsto & \sum_{j \in \N} a_j x_j
\end{array} \right. \]
is a quotient map with the property that, for each $\varepsilon > 0$, $B_X \subset Q((1+ \varepsilon) B_Z)$.
\end{prop}

\begin{proof}
$(i)$ Let $\sum_{j \in \N} a_j e_j \in c_{00}$, and $n \in \N$. We define $b= \sum_{j=1}^n a_j e_j$. We choose intervals $I_1 < \cdots <I_k$ in $\N$ so that
\[ \| b \|_Z = \Big\| \sum_{j=1}^k \| \sum_{i \in I_j} b_i x_i \|_X u_{\min I_j} \Big\|_U . \]
Let us note that we can assume $\max I_k \leq n$. Now, for all $j \leq k$, we have $\sum_{i \in I_j} b_i e_i = \sum_{i \in I_j} a_i e_i$ so $\|a\|_Z \geq \|b\|_Z$. \\
$(ii)$ Let $(z_l)$ be a block in $Z$, say $z_l= \sum_{i=n_{l-1}+1}^{n_l} a_i e_i$, with $n_0=0<n_1 < n_2 < \cdots$, and $\|z_l\|_Z \geq 1$, for $l \in \N$. We need to show that 
\[ \liminfty \Big\| \sum_{i=1}^n z_i \Big\|_Z = \infty . \]
Using Lemma \ref{lemma 3}, we can choose natural numbers $0=k_0<k_1 < k_2 < \cdots$ and, for every $l \in \N$, intervals
\[ n_{l-1}+1 \leq I_{k_{l-1}+1} < I_{k_{l-1}+2} < \cdots < I_{k_l} \leq n_l, \]
so that
\[ \|z_l\|_Z \leq 2 \Big\| \sum_{j=k_{l-1}+1}^{k_l} \|\sum_{i \in I_j} a_i x_i \|_X u_{\min I_j} \Big\|_U . \]
Define $y_l= \sum_{j=k_{l-1}+1}^{k_l} \|\sum_{i \in I_j} a_i x_i \|_X u_{\min I_j} \in U$ for every $l \in \N$. Since $(u_i)_{i \in \N}$ is assumed to be boundedly complete, it follows from our assumption on $(z_l)$ that
\[ \liminfty \Big\| \sum_{l=1}^n y_l \Big\|_U = \infty . \]
But this implies that
\begin{align*}
\Big\| \sum_{i=1}^n z_i \Big\|_Z & \geq \Big\| \sum_{j=1}^{k_n} \|\sum_{i \in I_j} a_i x_i \|_X u_{\min I_j} \Big\|_U \\
&= \Big\| \sum_{l=1}^n y_l \Big\|_U \nearrow \infty, \text{ if } n \nearrow \infty .
\end{align*}
$(iii)$ For $a= \sum_{j \in \N} a_j e_j \in c_{00}$, it follows that 
\[ \|a\|_Z \geq \| \sum_{i \in \N} a_i x_i \|_X . \]
Thus, the operator 
\[ Q : \left\lbrace \begin{array}{lll}
Z & \to & X \\
\sum_{i \in \N} a_i e_i & \mapsto & \sum_{i \in \N} a_i x_i
\end{array} \right. \]
is well defined and $\|Q\| \leq 1$. Let $x \in B_X$ and $\varepsilon > 0$. Inductively, we can choose, for each $j \in \N$, $\sigma_j = \pm 1$, $n_j \in \N$, $a_j \in (0, 2^{-j} \varepsilon )$, if $j \geq 2$, $a_1=\|x\|_X$, so that $n_1 < n_2 < \cdots$, and 
\[ \Big\| x- \sum_{j=1}^l \sigma_j a_j x_{n_j} \Big\|_X < \varepsilon 2^{-l} . \]
Then
\[ z= \sum_{j=1}^{\infty} \sigma_j a_j e_{n_j} \in (1+ \varepsilon) B_Z \]
and 
\[ Q(z)= \lim\limits_{l \to \infty} \sum_{j=1}^l \sigma_j a_j x_{n_j} = x . \]
\end{proof}

From now on, we assume that $(u_j)_{j \in \N}$ is boundedly complete. By Proposition \ref{prop 2} $(ii)$, $Z$ is then the dual of the space $Y = \overline{\vspan(e_j^* ; j \in \N)} \subset Z^*$, where $(e_j^*)_{j \in \N}$ are the coordinate functionals of $(e_j)_{j \in \N}$ (cf Proposition 1.b.4 of \cite{LT1}). Therefore, $(e_j^*)_{j \in \N}$ is a shrinking basis of $Y$. For $z= \sum_{i \in \N} a_i e_i$, we call the set
\[ \supp(z)=\supp_Z(z)=\{ i \in \N ; a_i \neq 0 \} \]
the support of $z$. For $z^* \in Z^*$, we call
\[ \supp(z^*)=\supp_{Z^*}(z^*) = \{ i \in \N ; z^*(e_i \neq 0 \} \]
the support of $z^*$. \\
From Proposition \ref{prop 2} $(iii)$, it follows that $Q^* : X^* \to Z^*$ is an isometric embedding.

\begin{lemma}
Considering $Q^*(X^*)$ as a subspace of $Z^*$, it follows that $Q^*(X^*) \cap Y = \{0\}$ and $Q^*(X^*) + Y$ is norm closed.
\end{lemma}

\begin{proof}
Let $x^* \in X^*$ and $z^* \in Y \subset Z^*$. Since the $(e_n)_{n \in \N}$, acting on elements of $Y$, are the coordinate functionals of $(e_n^*)_{n \in \N}$, it follows that $\lim_{n \to \infty} z^*(e_n)=0$. Thus, by the density condition on the $x_n$, $n \in \N$, it follows that
\[ \|Q^*(x^*)+z^*\| \geq \limsup\limits_{n \to \infty} |x^*(x_n)+z^*(e_n)|=\limsup_{n \to \infty} |x^*(x_n)|=\|x^*\|=\|Q^*(x^*)\|. \]
Since $x^* \in X^*$ and $z^* \in Y \subset Z^*$ are arbitrary, it follows that $Q^*(X^*) \cap Y = \{0\}$. Secondly, it follows that $Q^*(X^*) + Y$ is the direct sum of $Q^*(X^*)$ and $Y$, and thus $Q^*(X^*) + Y$ is norm closed.
\end{proof}

For $A \subset \N$ finite or cofinite, we denote by $P_A : Z \to Z$ the canonical projection onto $\overline{\vspan(e_j ; j \in A)}$, \ie
\[ P_A : \left\lbrace \begin{array}{lll}
Z & \to & Z \\
\sum_{i \in \N} a_i e_i & \mapsto & \sum\limits_{i \in A} a_i e_i
\end{array} \right. . \]
$P_A^*$ denotes the adjoint of $P_A$.

\begin{lemma} \label{lemma 5}
Define
\begin{align*}
S &= \left\{ \sum_{j=1}^k P^*_{I_j} \circ Q^*(x_j^*) \big| \begin{matrix}
I_1 < \cdots < I_k \text{ intervals in $\N$, } \min(I_{j+1})=\max(I_j)+1, \text{ if $j <k$} \\
x_j^* \in X^*, j \leq k, \text{ with } \| \sum_{j=1}^k \|x^*_j \| u_{\min(I_j)}^* \|_{U^*} \leq 1
\end{matrix}  \right\} \\
&= \left\{ \sum_{j=1}^k P^*_{[r_j, r_{j+1})} \circ Q^*(x_j^*) \big| \begin{matrix}
1 \leq r_1 < \cdots < r_k < r_{k+1} \text{ and } \\
x_j^* \in X^*, j \leq k, \text{ with } \| \sum_{j=1}^k \|x^*_j \| u_{r_j}^* \|_{U^*} \leq 1
\end{matrix} \right\} .
\end{align*}
Then $S \subset B_Y \subset B_{Z^*}$ and $S$ is norming the elements of $Z$, and thus its weak$^*$-closed convex hull is $B_{Z^*}$ by a Hahn-Banach argument. \\
Secondly, let $m \leq n$ and $z^* \in \vspan(e_j^* ; m \leq j \leq n)$ then it can be written as 
\begin{equation} \label{rep}
z^*= \sum_{i=1}^l \sum_{j=1}^{k_i} P^*_{[r_{(i,j)}, r_{(i,j+1)})} \circ Q^*(x^*_{(i,j)}),
\end{equation}
where $l \in \N$, $m \leq r_{(i,1)} < \cdots < r_{(i, k_i+1)} \leq n+1$ and $x_{(i,j)}^* \in X^*$, $j \leq k_i$, $i \leq l$. \\
And $\|z^*\|$ is at most the minimum of $\sum_{i=1}^l \| \sum_{j=1}^{k_i} \|x^*_{(i,j)} \| u^*_{r_{(i,j)}} \|_{U^*}$ over all representations of $z^*$ as in \eqref{rep}.
\end{lemma}

\begin{proof}
Let $z= \sum_{i \in \N} a_i e_i \in Z$ and $z^*= \sum_{j=1}^k P^*_{[r_j,r_{j+1})} \circ Q^*(x_j^*) \in S$. Then
\begin{align*}
z^*(z) &= \left( \sum_{j=1}^k P^*_{[r_j,r_{j+1})} \circ Q^*(x^*_j) \right) \left( \sum_{i=1}^n a_i e_i \right) \\
&= \sum_{j=1}^k Q^*(x_j^*) \left( \sum_{i=r_j}^{r_{j+1}-1} a_i e_i \right) \\
&= \sum_{j=1}^k \sum_{i=r_j}^{r_{j+1}-1} a_i x_j^*(x_i) \\
& \leq \sum_{j=1}^k \|x_j^*\| \cdot \Big\| \sum_{i=r_j}^{r_{j+1}-1} a_i x_i \Big\| \\
&= \left( \sum_{j=1}^k \|x_j^*\| u_{r_j}^* \right) \left( \sum_{j=1}^k \Big\| \sum_{i=r_j}^{r_{j+1}-1} a_i x_i \Big\| u_{r_j} \right) \\
& \leq \Big\| \sum_{j=1}^k \| \sum_{i=r_j}^{r_{j+1}-1} a_i x_i \| u_{r_j} \Big\|_U \leq \|z\|_Z,
\end{align*}
which proves that $S \subset B_{Z^*}$. Since, for $i \in \N$, $e_i^*$ can be written as $e_i^*= P_i^* \circ Q^*(x_i^*)$, where $x_i^* \in S_{X^*}$ with $x_i^*(x_i)=1$, it follows that, for $m \leq n$, every $z^* \in \vspan(e_i^* ; m \leq j \leq n)$ can be represented as in \eqref{rep} and in particular, it follows that $S \subset B_Y$. \\
For $z = \sum_{i \in \N} a_i e_i \in Z \cap c_{00}$, it follows, for an appropriate choice of $k$ and $r_1 < \cdots < r_{k+1}$ in $\N$ that
\begin{align*}
\|z\| &= \Big\| \sum_{j=1}^k \Big\| \sum_{i=r_j}^{r_{j+1}-1} a_i x_i \Big\|_X u_{r_j} \Big\|_U \\
&= \sup \left\{ \sum_{j=1}^k b_j \Big\| \sum_{i=r_j}^{r_{j+1}-1} a_i x_i \Big\|_X ; \Big\| \sum_{j=1}^k b_j u_{r_j}^* \Big\|_{U^*} \leq 1 \right\} \\
&= \sup \left\{ \sum_{j=1}^k b_j x_j^* \Big(\sum_{i=r_j}^{r_{j+1}-1} a_i x_i \Big) ; x_j^* \in S_{X^*}, j \in \{1, \cdots, k \}, \text{ and } \Big\| \sum_{j=1}^k b_j u_{r_j}^* \Big\|_{U^*} \leq 1 \right\} \\
&= \sup \left\{ \sum_{j=1}^k x_j^* \Big(\sum_{i=r_j}^{r_{j+1}-1} a_i x_i \Big) ;  \Big\| \sum_{j=1}^k \|x_j^*\|_{X^*} u_{r_j}^* \Big\|_{U^*} \leq 1 \right\} \\
&= \sup \left\{ \Big( \sum_{j=1}^k P^*_{[r_j,r_{j+1}-1)} \circ Q^*(x_j^*) \Big)(z) ; \Big\| \sum_{j=1}^k \|x_j^*\|_{X^*} u_{r_j}^* \Big\|_{U^*} \leq 1 \right\} \\
& \leq \sup\limits_{z^* \in S} z^*(z),
\end{align*}
which implies that $S$ is norming the elements of $Z$, and thus its weak$^*$-closed convex hull is $B_{Z^*}$. \\
For $z^*= \sum_{i=1}^l \sum_{j=1}^{k_i} P^*_{[r_{(i,j)},r_{(i,j+1)})} \circ Q^*(x^*_{(i,j)}) \in Y$ and $z= \sum_{i \in \N} a_i e_i \in Z$, it follows that
\begin{align*}
z^*(z) &= \left( \sum_{i=1}^k \sum_{j=1}^{k_i} P^*_{[r_{(i,j)},r_{(i,j+1)})} \circ Q^*(x^*_{(i,j)}) \right) \left( \sum_{i \in \N} a_i e_i \right) \\
&= \sum_{i=1}^k \sum_{j=1}^{k_i} x^*_{(i,j)} \left( \sum_{s=r_{(i,j)}}^{r_{(i,j+1)}-1} a_s x_s \right) \\
& \leq \sum_{i=1}^k \sum_{j=1}^{k_i} \|x^*_{(i,j)}\|_{X^*} \Big\| \sum_{s=r_{(i,j)}}^{r_{(i,j+1)}-1} a_s x_s \Big\|_X \\
&= \sum_{i=1}^k \left( \sum_{j=1}^{k_i}  \|x^*_{(i,j)}\|_{X^*} u^*_{r_{(i,j)}} \right) \left( \sum_{j=1}^{k_i} \Big\| \sum_{s=r_{(i,j)}}^{r_{(i,j+1)}-1} a_s x_s \Big\|_X u_{r_{(i,j)}} \right) \\
& \leq \sum_{i=1}^k \Big\| \sum_{j=1}^{k_i} \|x^*_{(i,j)}\|_{X^*} u^*_{r_{(i,j)}} \Big\|_{U^*} \|z\|_Z,
\end{align*}
which implies that
\[ \|z^*\|_{Z^*} \leq \sum_{i=1}^k \Big\| \sum_{j=1}^{k_i} \|x^*_{(i,j)}\|_{X^*} u^*_{r_{(i,j)}} \Big\|_{U^*} , \]
and thus our claim on the upper estimate of $\|z^*\|_{Z^*}$.
\end{proof}

Our next goal is to formulate a condition on the space $U$ under which $Z^*$ is equal to $Q^*(X^*) \oplus Y$. We first make some general observations. \\
Assume that $z^* \in Z^*$ with $\operatorname{dist}(z^*,Y \oplus Q^*(X^*)) > 0$, for some $z^* \in B_{Z^*}=B_{Y^{**}}$. By Lemma \ref{lemma 5}, we can write $z^*$ as
\[ z^*= w^* - \liminfty z_n^*, \text{ with } z_n^*= \sum_{l=1}^{s_n} \lambda_{(n,l)} y^*_{(n,l)} \text{ where} \]
\[ \lambda_{(n,l)} \geq 0 \text{ for } l \leq s_n \text{ and } n \in \N, \text{ and } \sum_{l=1}^{s_n} \lambda_{(n,l)}=1 \text{ for } l \in \N \]
\[ y^*_{(n,l)}= \sum_{j=1}^{k_{(n,l)}} P^*_{[r_{(n,l,j)},r_{(n,l,j+1)})} \circ Q^*(x^*_{(n,l,j)}) \in S. \]
By possibly adding $0$, we assume that $1=r_{(n,l,1)} < \cdots < r_{(n,l,k_{(n,l)})} < r_{(n,l,k_{(n,l)}+1)} = \infty$. \\
For $n \in \N$, and $l \leq s_n$, we define $R_{(n,l)}=\{ r_{(n,l,j)}, 1 \leq j \leq k_{(n,l)} \}$ and, for $p<q$, we put
\[ A_{(p,q,n)}= \{ l \leq s_n ; (p,q) \cap R_{(n,l)} = \varnothing \} = \{ l \leq s_n ; \exists 1 \leq j \leq k_{(n,l)} \quad r_{(n,l,j)} \leq p < q \leq r_{(n,l,j+1)} \}  . \]
We define
\[ a_n^*(p,q)= \sum_{l \in A_{(p,q,n)}} \lambda_{(n,l)} y^*_{(n,l)}  . \]
By passing to subsequences and using diagonalization, we can assume that, for all $p<q$
\[ \varepsilon(p,q)= \liminfty \sum_{l \in A_{(p,q,n)}} \lambda_{(n,l)} \]
exists. \\
Note that, for fixed $p \in \N$, $\varepsilon(p,q)$ is decreasing (not necessarily strictly) to some $\varepsilon(p) \geq 0$, if $q$ increases to $\infty$, and $\varepsilon(p)$ is increasing to some $\varepsilon_0$, if $p$ increases to $\infty$. \\
We first consider the following case: \\
$\bullet$ Case 1. $\varepsilon_0 > 0$, which means that there is a $p \in \N$ and an $\varepsilon > 0$ so that, for all $p <q$,
\[ \liminfty \sum_{l \in A_{(p,q,n)}} \lambda_{(n,l)} > \varepsilon . \]
First, fix $p \in \N$ so that $\varepsilon(p) > 0$. After passing to a subsequence, we can assume that there is an increasing sequence $(q_n)_{n \in \N}$ in $\N$ so that
\[ \sum_{l \in A_{(p,q_n,n)}} \lambda_{(n,l)} > \varepsilon(p) \left( 1- \frac{1}{n} \right), \text{ for all } n \in \N . \]
After passing to a subsequence, we can assume that $v_p^*=w^* - \lim_{n \to \infty} a^*_{(p,q_n)}$ exists. \\
Let $n \in \N$. For $l \in A_{(p,q_n,n)}$, choose $j_l \in \{1, \cdots, k_{(n,l)} \}$, so that $[p,q_n] \subset [r_{(n,l,j_l)},r_{(n,l,j_l+1)}]$ and put
\[ x_n^* = \sum_{l \in A_{(p,q_n,n)}} \lambda_{(n,l)} x^*_{(n,l,j_l)} . \]
Then, $\|x_n^*\| \leq \sum_{l \in A_{(p,q_n,n)}} \lambda_{(n,l)}$ and $a_n^*(p,q_n)(e_i)=x_n^*(x_i)$, for all $i \in [p,q_n]$. \\
Again, by passing to a subsequence, we can assume that $x^*=w^* - \lim_{n \to \infty} x_n^*$ exists and
\begin{equation}
\|x^*\| \leq \liminf_{n \to \infty} \|x_n^*\| \leq \liminf_{n \to \infty} \sum_{l \in A_{(p,q_n,n)}} \lambda_{(n,l)} = \varepsilon(p) . 
\end{equation}
It follows, for all $i \geq p$, that
\[ v_p^*(e_i)= \liminfty a_n^*(p,q_n)(e_i)= \liminfty x_n^*(x_i)=x^*(x_i) . \]
Thus, we can write
\[ v_p^*=Q^*(x^*)-P^*_{[1,p)} \circ Q^*(x^*) + P^*_{[1,p)}(v_p^*)=P^*_{[p, \infty ]} \circ Q^*(x^*) + P_{[1,p)}^*(v_p^*), \]
which implies that $v_p^* \in Q^*(X^*) \oplus Y$. \\
It follows that $\operatorname{dist}(z^*-v_p^*, Q^*(X^*) \oplus Y)= \operatorname{dist}(z^*, Q^*(X^*) \oplus Y) > 0$.\\
Now, we can iterate the above argument applying it to large enough $p \in \N$ and find $p_0 \in \N$ and, for every $p \in \N$, $p \geq p_0$, a vector $v_p^* \in Q^*(X^*) \oplus Y$, an increasing sequence $(q_n)_{n \in \N}$ in $\N$ and infinite sets $\N \supset N_{p_0} \supset N_{p_0+1} \supset \cdots$ so that
\[ \operatorname{dist}(z^*-v_p^*, Q^*(X^*) \oplus Y)= \operatorname{dist}(z^*, Q^*(X^*) \oplus Y) \]
and
\[ v_p^* = w^* - \lim\limits_{n \to \infty, n \in N_p} \quad \sum_{l \in A_{(p,q_n,n)}} \lambda_{(n,l)} y^*_{(n,l)} .  \]
Since, for $p < p'$
\[ v_{p'}^*-v_p^* = w^* - \lim\limits_{n \to \infty, n \in N_{p'}} \quad \sum_{l \in A_{(p',q_n,n)} \setminus A_{(p,q_n,n)}} \lambda_{(n,l)} y^*_{(n,l)} \]
it follows that $(v_p^*)_{p \in \N}$ is a (norm) Cauchy sequence in $Q^*(X^*) \oplus Y$ which converges to some $v^*$. \\
Since the $w^*$-topology is metrizable, it follows that we can assume (after possibly passing to a subsequence of $(q_n)_{n \in \N}$), that, for some diagonal sequence $N$ of the $N_p$, $p \in \N$, and some increasing sequence $(p_n)_{n \in \N}$, we have
\[ \varepsilon_0 = \lim\limits_{n \to \infty, n \in N} \quad \sum_{l \in A_{(p_n,q_n,n)}} \lambda_{(n,l)} , \]
\[ v^* = w^* - \lim\limits_{n \to \infty, n \in N} \quad \sum_{l \in A_{(p_n,q_n,n)}} \lambda_{(n,l)} y^*_{(n,l)} \in Y \oplus Q^*(X^*) , \]
\[ \operatorname{dist}(z^*-v^*, Q^*(X^*) \oplus Y)= \operatorname{dist}(z^*, Q^*(X^*) \oplus Y) > 0 . \]
Now consider
\[ \tilde{z}^* = z^*-v^*= w^* - \lim\limits_{n \to \infty, n \in N} \tilde{z}_n^*  , \]
where, for $n \in N$, we put
\[ \tilde{z}_n^* = \sum_{l \in \{1, \cdots, s_n \} \setminus A_{(p_n,q_n,n)}} \lambda_{(n,l)} y^*_{(n,l)} \]
and define, for $n \in N$, and $p < q$
\[ B_{(p,q,n)} = \left\{ l \leq s_n, l \notin A_{(p_n,q_n,n)} \text{ and } \exists 1 \leq j \leq k_{(n,l)} \quad r_{(n,l,j)} \leq p < q \leq r_{(n,l,j+1)} \right\} . \]
It follows that, for any $p \in \N$, we have
\[ \lim\limits_{n \to \infty, n \in N} \quad \sum_{l \in B_{(p,q_n,n)}} \lambda_{(n,l)} \leq \lim\limits_{n \to \infty, n \in N} \quad \sum_{l \in B_{(p_n,q_n,n)}} \lambda_{(n,l)} = 0 . \]
Replacing now $z^*$ by $\tilde{z}^*$, and replacing the sequence $(z_n^*)_{n \in \N}$ by $(\tilde{z}_n^*)_{n \in \N}$, we can from now on assume that we are in Case $2$. \\
$\bullet$ Case $2$. For all $\varepsilon > 0$ and all $p \in \N$, there is a $q$ so that
\[ \limsup_{n \to \infty} \sum_{l \in A_{(p,q,n)}} \lambda_{(n,l)} < \varepsilon \]
($A(p,q,n)$ defined as above). \\
Let $\eta < \frac{1}{2} \operatorname{dist}(z^*, Q^*(X^*) \oplus Y)$ and let $(\eta_m)_{m \in \N} \subset (0, \eta)$, with $\sum_{m \in \N} \eta_m < \eta$. Inductively, we can choose $1=q_1 < q_2 < \cdots$ so that
\[ \limsup_{n \to \infty} \sum_{l \in A_{(q_m,q_{m+1},n)}} \lambda_{(n,l)} < \eta_m \text{ for all } m \in \N . \]
Put, for $n \in \N$:
\[ \tilde{z}_n^*= \sum_{l \in \{1, \cdots, s_n \} \setminus \bigcup_{m} A_{(q_m,q_{m+1},n)}} \lambda_{(n,l)} y^*_{(n,l)}  \]
then $\|z_n^*-\tilde{z}_n^* \| < \eta$. After passing to a subsequence, we can assume that $\tilde{z}^* = w^*- \lim_{n \to \infty} \tilde{z}_n^*$ exists. It follows that $\|\tilde{z}^*\| \leq 1$ and $\|z^*-\tilde{z}^*\| \leq \eta$, and thus
\[ \operatorname{dist}(\tilde{z}^*, Y \oplus Q^*(X^*)) = \operatorname{dist}(z^*, Y \oplus Q^*(X^*)) - \eta > 0 . \] 
We replace from now on $z^*$ by $\tilde{z}^*$ and $z_n^*$ by $\tilde{z}_n^*$, and therefore assume that, for each $n \in \N$ and each $l \leq s_n$, there is a $t_{(n,l)} \leq k_{(n,l)}$ and a subsequence $(j_l)_{l=1}^{t_{(n,l)}}$ of $\{1, \cdots, k_{(n,l)} \}$, $j_1=1$, so that
\begin{align*}
1= & r_{(n,l,1)} < r_{(n,l,2)} < \cdots < r_{(n,l,j_2-1)} < q_2 \leq r_{(n,l,j_2)} < \cdots < r_{(n,l,j_3-1)} < q_3 \leq r_{(n,l,j_3)} < \cdots \\
& < q_{t_{(n,l)}-1} \leq r_{(n,l,j_{t_{(n,l)}-1})} < \cdots < r_{(n,l,j_{t_{(n,l)}}-1)} < q_{t_{(n,l)}} \leq r_{(n,l,t_{(n,l)})} < \dots < q_{t_{(n,l)}+1} .
\end{align*}
After possibly adding some zero vectors to the $y^*_{(n,l)}$, we can assume that, for all $n \in \N$ and all $l \leq k_{(n,l)}$, $t_{(n,l)}=t_n= \max_{l \leq s_n} t_{(n,l)}$. We can also assume that $t_n$ is even, say $t_n=2 w_n$, and that $w_n \geq n$. \\
We write now, for $n \in \N$, $z_n^*=f_n^*+g_n^*$ where
\[ f_n^*= \sum_{l=1}^{s_n} \lambda_{(n,l)} \sum_{i=1}^{w_n} \sum_{j=j_{2i-1}}^{j_{2i}-1} P^*_{[r_{(n,l,j)},r_{(n,l,j+1)})} \circ Q^*(x^*_{(n,l,j)}) \text{ and } \]
\[ g_n^* = \sum_{l=1}^{s_n} \lambda_{(n,l)} \sum_{i=1}^{w_n} \sum_{j=j_{2i}}^{j_{2i+1}-1} P^*_{[r_{(n,l,j)},r_{(n,l,j+1)})} \circ Q^*(x^*_{(n,l,j)})  . \]
(here, we put $j_{2i+1}-1=k_{(n,l)}$ if $i=w_n$). \\
After passing to a subsequence, we can assume that $f^*=w^* - \lim_{n \to \infty} f_n^*$ and $g^*= w^* - \lim_{n \to \infty} g_n^*$ exist. It also follows from the above representations of the $f_n^*$ and $g_n^*$ that $\|f^*\|$, $\|g^*\| \leq 1$. Since $z^*=f^*+g^*$, we can, without loss of generality, assume that $\operatorname{dist}(f^*, Y \oplus Q^*(X^*)) > 0$. \\
So, again, after replacing $z^*$ by $f^*$, and defining $r_j=q_{2j-1}$, for $j \in \N$, $\overline{r}=(r_j)_{j \in \N}$ and possibly adding again zero vectors, we can assume that, for each $n \in \N$, the vectors $z_n^*$ are of the form
\[ z_n^*= \sum_{l=1}^{s_n} \lambda_{(n,l)} \sum_{j=1}^{w_n} v^*_{(n,l,j)} \in S, \]
where
\[ v^*_{(n,l,j)} = \sum_{i=\tilde{m}_{(n,l,j)}}^{\tilde{m}_{(n,l,j+1)}-1} P^*_{[r_{(n,l,i)},r_{(n,l,i+1)})} \circ Q^*(x^*_{(n,l,i)}) \]
with $r_j=r_{(n,l,\tilde{m}_{(n,l,j)})} < r_{(n,l,\tilde{m}_{(n,l,j)}+1)} < \cdots < r_{(n,l,\tilde{m}_{(n,l,j+1)}-1)} < r_{(n,l,\tilde{m}_{(n,l,j+1)})}=r_{j+1}$, for $j \leq w_n$, $l \leq k_{(n,l)}$. \\
We define the following norm $\vvvert \cdot \vvvert_{\overline{r}}$ on $\vspan(e_j^* ; j \in \N)$. For $y^* \in \vspan(e_j^* ; j \in \N)$, we let $\vvvert y^* \vvvert_{\overline{r}}$ to be the infinimum of all expressions
\[ \sum_{i=1}^l \Big\| \sum_{j=1}^{k_i} \|x^*_{(i,j)}\|_{X^*} u^*_{r_{(i,j)}} \Big\|_{U^*}, \]
where the infinimum is taken over all representations of $y^*$ as
\[ y^* = \sum_{i=1}^l \sum_{j=1}^{k_i} P^*_{[r_{(i,j)},r_{(i,j+1)})} \circ Q^*(x^*_{(i,j)}) \]
where $l \in \N$, $r_{(i,1)} < \cdots < r_{(i,k_i)} < r_{(i,k_i+1)}$, and additionally, for some $m < n$ in $\N$,
\[ r_m=r_{(i,1)}, r_n=r_{(i,k_i+1)} \text{ and } \{ r_m, r_{m+1}, \cdots, r_n \} \subset \{ r_{(i,1)}, \cdots, r_{(i,k_i)}, r_{(i,k_i+1)} \} \]
and $x^*_{(i,j)} \in X^*$, $j \leq k_i$, $i \leq l$. \\
It follows that the norm $\vvvert \cdot \vvvert_{\overline{r}}$ dominates $\| \cdot \|_{Z^*}$ and that $\vvvert z_n^* \vvvert_{\overline{r}} \leq 1$ for all $n \in \N$. \\
For $n, j \in \N$, we put $v^*_{(n,j)}= \sum_{l=1}^{s_n} \lambda_{(n,l)} v^*_{(n,l,j)}$ if $j \leq w_n$ and $v_{(n,j)}^*=0$ if $j > w_n$. $v^*_{(n,j)}$ is an element of the finite-dimensional space $F_j = \vspan(e_i^* ; r_j \leq i < r_{j+1})$. We can therefore, after passing to a subsequence, assume that
\[ v_j^* = \vvvert \cdot \vvvert_{\overline{r}} - \liminfty v^*_{(n,j)} \in F_j \]
exists. It follows that
\[ \Big\vvvert \sum_{j=1}^m v_j^* \Big\vvvert_{\overline{r}} \leq 1 \text{ for every $m \in \N$, and } z^* = w^* - \lim\limits_{m \to \infty} \sum_{j=1}^m v_j^* . \]
Since $\sum_{j=1}^m v_j^*$ is not norm converging for $m \nearrow \infty$ (otherwise, $z^*$ would be an element of $Y$), it follows that we can find $\delta > 0$ (which can actually be chosen as close to $1$ as we wish) and an increasing sequence $(m_j)_{j=1}^{\infty}$, so that
\[ \Big\vvvert \sum_{i=m_j}^{m_{j+1}-1} v_i^* \Big\vvvert_{\overline{r}} \geq \delta . \]
We put $w_j^*= \sum_{i=m_j}^{m_{j+1}-1} v_i^*$, and we deduce that there is an increasing sequence $(k_n)_{n \in \N}$ in $\N$ and a sequence $(s_n)_{n \in \N}$ in $\N$ so that, for each $n \in \N$, $\Big\vvvert \sum_{j=1}^n w_j^* \Big\vvvert_{\overline{r}}$ can be written as
\begin{equation} \label{eq 3}
\Big\vvvert \sum_{j=1}^n w_j^* \Big\vvvert_{\overline{r}} = \sum_{l=1}^{s_n} \Big\| \sum_{j=1}^n \sum_{i=m_{(n,l,j)}}^{m_{(n,l,j+1)}-1} \|x^*_{(n,l,i)}\|_{X^*} u^*_{\tilde{r}_{(n,l,i)}} \Big\|_{U^*} \leq 1
\end{equation}
and, for every $j \in \N$, the vector $w_j^*$ has the representation
\begin{equation} \label{eq 4}
w_j^*= \sum_{l=1}^{s_n} \sum_{i=m_j}^{m_{j+1}-1} P^*_{[\tilde{r}_{(n,l,i)},\tilde{r}_{(n,l,i+1)})} \circ Q^*(x^*_{(n,l,i)}) 
\end{equation}
with
\begin{equation} \label{eq 5}
 \vvvert w_j^* \vvvert_{\overline{r}} \geq \delta,
\end{equation}
where, for $1 \leq l \leq s_n$ \\
$1=m_{(n,l,1)} < m_{(n,l,2)} < \cdots < m_{(n,l,n+1)}$ 
\begin{align*}
r_1= & \tilde{r}_{(n,l,1)} < \tilde{r}_{(n,l,2)} < \cdots < \tilde{r}_{(n,l,m_{(n,l,2)})}=r_{k_2} < \tilde{r}_{(n,l,m_{(n,l,2)}+1)} < \cdots < \tilde{r}_{(n,l,m_{(n,l,3)})}=r_{k_3} < \\
& \cdots < \tilde{r}_{(n,l,m_{(n,l,n)})}=r_{k_n} < \tilde{r}_{(n,l,m_{(n,l,n)}+1)} < \cdots < \tilde{r}_{(n,l,m_{(n,l,n+1)})}=r_{k_{n+1}}
\end{align*}
$x^*_{(n,l,i)} \in X^*$, for $l \leq s_n$ and $1 \leq i \leq m_{(n,l,n+1)}-1$, \\
(although $w_j^*$ does not depend on $n$, its specific representation to compute the norm of the sums of $w_j^*$ could depend on $n$). Note that \eqref{eq 4}, \eqref{eq 5}, and the definition of the norm $\vvvert . \vvvert_{\overline{r}}$ imply that, for $n \in \N$ and $j \in \N$
\begin{equation} \label{eq 6}
\sum_{l=1}^{s_n} \Big\| \sum_{i=m_{(n,l,j)}}^{m_{(n,l,j+1)}-1} \|x^*_{(n,l,i)}\|_{X^*} u^*_{\tilde{r}_{(n,l,i)}} \Big\|_{U^*} \geq  \vvvert w_j^* \vvvert_{\overline{r}} \geq \delta .
\end{equation}
This leads to the following result which provides a sufficient condition on $U$ and its dual to imply that $Z^*$ is the complemented sum of $Q^*(X^*)$ and $Y$.

\begin{lemma} \label{lemma 5.6}
Assume the basis $(u_j^*)_{j \in \N}$ of $U^*$ satisfies the following condition:
\begin{equation}
\begin{matrix}
\text{For every increasing sequence $\overline{M}=(M_n)_{n \in \N}$ in $\N$, there is a boundedly complete and }  \\ 
\text{$1$-unconditional basic sequence $(f_n)_{n \in \N}=(f_{(\overline{M},n)})_{n \in \N}$ (in some Banach space $F$) which} \\ 
\text{satisfies the following property:} \\
\text{Every normalized block sequence $(a_n^*)_{n \in \N}$ in $U^*$ with $\supp_{U^*}(a_n^*) \subset [M_n, M_{n+1})$, for all}  \\
\text{$n \in \N$ dominates $(f_n)_{n \in \N}$.} 
\end{matrix}
\end{equation}
Then $Z^*=Y \oplus Q^*(X^*)$.
\end{lemma}

\begin{proof}
Assume that our claim is wrong and that there is a $z^* \in B_{Z^*}$ which is not in $Q^*(X^*) \oplus Y$. By our previous arguments, we can assume that
\[ z^*= w^* - \liminfty \sum_{j=1}^n w_j^*, \]
where the sequence $(w_j^*)_{j \in \N}$ satisfies the above equations \eqref{eq 3}, \eqref{eq 4} and \eqref{eq 5}, for some $\delta > 0$, some sequences $(k_n)_{n \in \N}$, $(s_n)_{n \in \N}$, and families $(m_{(n,l,j)}; n \in \N, l \leq s_n, j \leq n+1)$, and $(\tilde{r}_{(n,l,j)} ; n \in \N, l \leq s_n, j \leq m_{(n,l,n+1)})$. \\
For $j \in \N$, we put $M_j = r_{k_j}=\tilde{r}_{(n,l,m_{(n,l,j)})}$ and let $(f_j)_{j \in \N}$ be a normalized $1$-unconditional basic sequence which is $C$-dominated by every normalized block sequence $(a_j^*)_{j \in \N}$ in $U^*$ with $\supp (a_j^*) \subset [M_j, M_{j+1})$, for each $j \in \N$. We deduce, for every $n \in \N$, that
\begin{align*}
1 & \geq \Big\vvvert \sum_{j=1}^n w_j^* \Big\vvvert_{\overline{r}} \\
&= \sum_{l=1}^{s_n} \Big\| \sum_{j=1}^n \sum_{i=m_{(n,l,j)}}^{m_{(n,l,j+1)}-1} \|x^*_{(n,l,i)}\|_{X^*} u_{\tilde{r}_{(n,l,i)}}^* \Big\|_{U^*} \text{ (By \eqref{eq 3})} \\
& \geq \frac{1}{C} \sum_{l=1}^{s_n} \Big\| \sum_{j=1}^n \Big\| \sum_{i=m_{(n,l,j)}}^{m_{(n,l,j+1)}-1} \|x^*_{(n,l,i)}\|_{X^*} u_{\tilde{r}_{(n,l,i)}}^* \Big\|_{U^*} f_j \Big\|_F \\
& \begin{pmatrix}
\text{For $l \in \{1, \cdots, s_n \}$, and $j \leq n$, put $a^*_{(l,j)}=\tilde{a}_{(l,j)} / \| \tilde{a}_{(l,j)} \|$ with} \\
\tilde{a}_{(l,j)} = \sum_{i=m_{(n,l,j)}}^{m_{(n,l,j+1)}-1} \| x^*_{(n,l,i)} \|_{X^*} u^*_{\tilde{r}_{(n,l,i)}}
\end{pmatrix} \\
& \geq \frac{1}{C} \Big\| \sum_{j=1}^n \sum_{l=1}^{s_n} \Big\| \sum_{i=m_{(n,l,j)}}^{m_{(n,l,j+1)}-1} \|x^*_{(n,l,i)}\|_{X^*} u_{\tilde{r}_{(n,l,i)}}^* \Big\|_{U^*} f_j \Big\|_F \\
& \geq \frac{1}{C} \Big\| \sum_{j=1}^n \vvvert w_j^* \vvvert_{\overline{r}} f_j \Big\|_F \geq \frac{\delta}{C} \Big\| \sum_{j=1}^n f_j \Big\|_F \text{ (By \eqref{eq 6} and unconditionality of $(f_j)_{j \in \N}$)}
\end{align*}
Since $(f_j)_{j \in \N}$ is boundedly complete, it follows that $\big\| \sum_{j=1}^n f_j \big\|_F \nearrow \infty$, if $n \nearrow \infty$, and thus we obtain a contradiction.
\end{proof}

Let us note that $T^*$ satisfies the hypothesis of Lemma \ref{lemma 5.6}. Indeed, by reflexivity and Lemma II.3 of \cite{CS}, we can chose $(f_n)_{n \in \N} = (u_{M_{n+1}-1}^*)_{n \in \N}$, where $(u_j^*)_{j \in \N}$ denotes the usual basis of $T^*$.

\begin{prop}
Assume that $U$ has the ``Tsirelson property", \ie, for some $C >0$, a normalized block of $n$ elements after $n$ is $C$ equivalent to the $\ell_1^n$-unit vector basis. \\
Then $Y$ is asymptotically $c_0$.
\end{prop}

\begin{proof}
Since $(e_j^*)_{j \in \N}$ is a shrinking basis of $Y$, it is enough to show that, for every $n \in \N$, every normalized block sequence $(y_j)_{j=1}^n$ in $Y$, with $n \leq \min \supp_Y (y_1)$ is $2C$-equivalent to the $\ell_{\infty}^n$-unit vector basis. So let $\varepsilon > 0$, $(a_j)_{j=1}^n \subset [-1,1]$ and choose $z= \sum_{i \in \N} b_i e_i \in S_Z \cap c_{00}$, so that
\[ z \left( \sum_{j=1}^n a_j y_j \right) \geq \Big\| \sum_{j=1}^n a_j y_j \Big\| - \varepsilon . \]
Let $r_j= \min \supp(y_j)$, for $1 \leq j \leq n$, $r_{n+1}= \infty$ and put $z_j= \sum_{i \in [r_j, r_{j+1)}} b_i e_i$. By Lemma \ref{lemma 3}, we can find, for each $j \in \{1, \cdots, n \}$, intervals $r_j \leq I_{(j,1)} < I_{(j,2)} < \cdots < I_{(j,k_j)} < r_{j+1}$ so that
\[ \|z_j\| \leq 2 \Big\| \sum_{l=1}^{k_j} \Big\| \sum_{i \in I_{(j,l)}} b_i x_i \Big\|_X u_{\min I_{(j,l)}} \Big\|_U . \]
By stringing the intervals to one sequence, it follows from our condition on $U$ that
\begin{align*}
1 &= \|z\|_Z \\
& \geq \Big\| \sum_{j=1}^n \sum_{l=1}^{k_j} \Big\| \sum_{i \in I_{(j,l)}} b_i x_i \Big\|_X u_{\min I_{(j,l)}} \Big\|_U \\
& \geq \frac{1}{C} \sum_{j=1}^n \Big\| \sum_{l=1}^{k_j} \Big\| \sum_{i \in I_{(j,l)}} b_i x_i \Big\|_X u_{\min I_{(j,l)}} \Big\|_U \\
& \geq \frac{1}{2C} \sum_{j=1}^n \|z_j\|_Z ,
\end{align*}
and thus
\[ \Big\| \sum_{j=1}^n a_j y_j \Big\|_Y \leq z \left( \sum_{j=1}^n a_j y_j \right) + \varepsilon = \sum_{j=1}^n a_j z_j(y_j)+\varepsilon \leq \sum_{j=1}^n \|z_j\|_Z + \varepsilon \leq 2C + \varepsilon, \]
which proves our claim, since $\varepsilon > 0$ was arbitrary.
\end{proof}

\begin{corollary}
If $U$ is $T$, the dual of the original Tsirelson space, as described in \cite{FJ}, with its usual $1$-unconditional basis, then $Z_X$ is asymptotic-$c_0$ and $Z_X^{**}$ is the complemented sum of $Z_X$ and $Q^*(X^*)$.
\end{corollary}

Define $W= \ker(Q) = \{z \in Z ; Q(z)=0 \}$. Since $Q^*(X^*)=W^{\perp}$, we deduce this final corollary by noticing that, in this case, $Y$ is isomorphic to $Z^* / W^{\perp}$, which is itself isomorphic to $W^*$.

\begin{corollary}
Assume that $Z^*$ is the complemented sum of $Y$ and $Q^*(X^*)$.  Then $Y$ is isomorphic to $W^*$.
\end{corollary}

\begin{thank}
The author would like to thank Gilles Lancien for useful conversations and comments, Florent Baudier for drawing our attention to Causey's paper \cite{3.5} and is very grateful to Thomas Schlumprecht for sending us his construction used for Theorem \ref{thm schlumprecht}, detailed below. The author also thanks the referee for her/his suggestions that helped improve this article.
\end{thank}

\bibliographystyle{plain}
\bibliography{biblio}

\end{document}